\documentclass[a4paper,12pt]{article}
\usepackage[dvipdfx,cmyk]{xcolor}
\definecolor{bookColor}{cmyk}{0 ,0 ,0 ,1} 
\color{bookColor} 
\usepackage{amsmath,amsthm,amssymb} 
\usepackage[hidelinks]{hyperref}
\usepackage{titlesec}
\titleformat{\section}
{\normalfont\Large\scshape\centering}{\thesection}{0.5em}{}

\titleformat{\subsection}
{\normalfont\large \scshape\centering}{\thesubsection}{0.3em}{}

\usepackage{xcolor}
\hypersetup{
	colorlinks,
	linkcolor={red},
	citecolor={blue}
}
\usepackage{graphicx}
\usepackage{caption}
\usepackage{float}
\usepackage{mathtools}
\usepackage{braket}
\usepackage{dsfont}

\usepackage{changepage}

\usepackage{marvosym}

\usepackage{tikz}
\usetikzlibrary{graphs,trees,automata,topaths,arrows,positioning}

\usepackage[toc,page,header]{appendix}

\DeclareMathOperator{\Aut}{Aut}

\DeclareMathOperator{\Sgn}{Sgn}
\DeclareMathOperator{\sgn}{sgn}
\DeclareMathOperator{\Ind}{Ind}
\DeclareMathOperator{\Res}{Res}

\DeclareMathOperator{\Sym}{Sym}
\DeclareMathOperator{\Alt}{Alt}
\DeclareMathOperator{\id}{id}

\DeclareMathOperator{\Stab}{Stab}
\newcommand{\N}{\mathbb N}
\newcommand{\C}{\mathbb C}

\newcommand{\tg}[1]{\textbf{#1}}

\newcommand{\es}{\varnothing}

\newcommand{\lb}{\lbrack}
\newcommand{\rb}{\rbrack}

\newcommand{\s}[2]{\sum\limits_{#1}^{#2}}

\newcommand{\prods}[2]{\langle #1,#2\rangle}

\newcommand{\restr}[2]{{% we make the whole thing an ordinary symbol
		\left.\kern-\nulldelimiterspace % automatically resize the bar with \right
		#1 % the function
		\vphantom{\big|} % pretend it's a little taller at normal size
		\right|_{#2} % this is the delimiter
	}}

\newcommand{\fct}[4]{\qq:\qq #1\qq\longrightarrow\qq #2\qq:\qq #3\qq \mapsto\qq #4}

\newcommand{\q}{\quad}
\newcommand{\qq}{\mbox{ }}

\newcommand{\diff}{\text{\rm d}}

\newcommand{\Hr}[1]{\mathcal{H}_{#1}}

\newcommand{\Fix}{\text{\rm Fix}}

\def\restriction#1#2{\mathchoice
	{\setbox1\hbox{${\displaystyle #1}_{\scriptstyle #2}$}
		\restrictionaux{#1}{#2}}
	{\setbox1\hbox{${\textstyle #1}_{\scriptstyle #2}$}
		\restrictionaux{#1}{#2}}
	{\setbox1\hbox{${\scriptstyle #1}_{\scriptscriptstyle #2}$}
		\restrictionaux{#1}{#2}}
	{\setbox1\hbox{${\scriptscriptstyle #1}_{\scriptscriptstyle #2}$}
		\restrictionaux{#1}{#2}}}
\def\restrictionaux#1#2{{#1\,\smash{\vrule height 1.2\ht1 depth 1.0\dp1}}_{\,#2}} 

\theoremstyle{plain}
\newtheorem{theorem}{Theorem}
\newtheorem{proposition}[theorem]{Proposition}
\newtheorem{lemma}[theorem]{Lemma}

\newtheorem{theoremletter}{Theorem}

\theoremstyle{definition}
\newtheorem{definition}[theorem]{Definition}
\newtheorem{remark}[theorem]{Remark}

\theoremstyle{plain}
\newtheorem*{theorem*}{Theorem}
\newtheorem*{conjecture*}{Conjecture}
\newtheorem*{proposition*}{Proposition}
\newtheorem*{lemma*}{Lemma}
\newtheorem*{corollary*}{Corollary}

\theoremstyle{definition}
\newtheorem*{definition*}{Definition}
\newtheorem*{remark*}{Remark}
\newtheorem*{example*}{Example}

\numberwithin{theorem}{section}
\numberwithin{equation}{section}
\begin{document}

	\title{Radu groups acting on trees are CCR}
	\author{Lancelot Semal\footnote{
			F.R.S.-FNRS Research Fellow, email : lancelot.semal\MVAt uclouvain.be}\\
			\newline \\
		UCLouvain, 1348 Louvain-la-Neuve, Belgium}
	
	\maketitle
	
	\begin{abstract}
		We classify the irreducible unitary representations of closed simple groups of automorphisms of trees acting $2$-transitively on the boundary and whose local action at every vertex contains the alternating group. As an application, we confirm Claudio Nebbia's CCR conjecture on trees for $(d_0,d_1)$-semi-regular trees such that $d_0,d_1\in \Theta$, where $\Theta$ is an asymptotically dense set of positive integers. 
	\end{abstract}
	
	\section{Introduction}
	In this document topological groups are second-countable, locally compact groups are Hausdorff and the word ``representation" stands for strongly continuous unitary representation on a separable complex Hilbert space. A locally compact group  $G$ is called \textbf{CCR} if the operator $\pi(f)$ is compact for all irreducible representation $\pi$ of $G$ and all $f\in L^{1}(G)$. For totally disconnected locally compact groups this property is equivalent to ask that every irreducible representation of $G$ is admissible see \cite{Nebbia1999}. We recall that an irreducible representation $\pi$ of totally disconnected locally compact group $G$ is \textbf{admissible} if for every compact open subgroup $K \leq G$ the space $\Hr{\pi}^K$ of $K$-invariant vectors is finite dimensional. A very important property of CCR groups is that they are \textbf{type I} groups \cite[Definition 6.E.7. and Proposition 6.E.11]{BekkadelaHarpe2020}. Loosely speaking, type I groups are the locally compact groups all of whose representations can be written as a unique direct integral of irreducible representations, thus reducing the study of arbitrary  representations to considerations on irreducible representations. Concerning groups of automorphisms of trees, Nebbia's work highlighted surprising relations between the action on the boundary and the regularity of representation theory. To be more precise, he showed in \cite{Nebbia1999} that any closed unimodular CCR subgroup $G\leq \Aut(T)$ of the group of automorphisms of a regular tree $T$ necessarily acts transitively on the boundary $\partial T$ . Further progress going in that direction were recently achieved by Houdayer and Raum \cite{HoudayerRaum2019} and with higher level of generality by Caprace, Kalantar and Monod \cite{CapraceKalantarMonod2022}. Among other things, they showed that a closed non-amenable type I subgroup acting minimally on a locally finite tree $T$ acts $2$-transitively on the boundary $\partial T$ \cite[Corollary D]{CapraceKalantarMonod2022}. Going in the other direction, Nebbia conjectured in \cite{Nebbia1999} that any closed subgroup of automorphisms of a regular tree acting transitively on the boundary is CCR. His conjecture naturally extends to the case of semi-regular trees.
	\begin{conjecture*}[CCR conjecture on trees \cite{Nebbia1999}]
		Let $T$ be a $(d_0,d_1)$-semi-regular tree with $d_0,d_1\geq 3$ and let $G\leq \Aut(T)$ be a closed subgroup acting transitively the boundary of $T$. Then $G$ is CCR. 
	\end{conjecture*}
	We recall from \cite[Lemma 3.1.1]{BurgerMozes2000} that, for a locally finite tree $T$, closed subgroups $G\leq \Aut(T)$ are non-compact and act transitively on the boundary $\partial T$ if and only if they act $2$-transitively on $\partial T$. Furthermore, the existence of such a group implies that the tree is semi-regular. In particular, since compact groups are automatically CCR we deduce that the hypothesis of semi-regularity is non-restrictive in the conjecture.
	
	One of the first evidence supporting the conjecture was provided by Bernstein and Harish-Chandra's works. Among other things, they proved that rank one semi-simple algebraic groups over local-fields are uniformly admissible \cite{Bernshtein1974}, \cite{Harish1970}. We recall that a totally disconnected locally compact group $G$ is \textbf{uniformly admissible} if for every compact open subgroup $K$, there exists a positive integer $k_K$ such that $\dim(\Hr{\pi}^K)<k_K$ for all irreducible representation $\pi$ of $G$. In particular, uniformly admissible groups are CCR. Concerning non-linear groups, the conjecture was supported by the complete classification of the irreducible representations of the full group of automorphisms of a semi-regular tree and more generally of closed subgroups acting transitively on the boundary and satisfying the Tits independence property  \cite{Olshanskii1977}, \cite{Olshanskii1980}, \cite{Choucroun1994}, \cite{FigaNebbia1991}, \cite{Amann2003} (those classifications lead to the conclusion that they are uniformly admissible \cite{Ciobotaru2015}). 
	
	Our paper concerns closed subgroups acting $2$-transitively on the boundary $\partial T$ and whose local action at every vertex $v$ contains the alternating group of corresponding degree. We recall that for each vertex $v\in V(T)$, the stabilizer $\Fix_G(v)$ of $v$ acts on the set $E(v)$ of edges containing $v$. The image of $\Fix_G(v)$ in $\Sym(E(v))$ for this natural projection map is called the \textbf{local action} of $G$ at $v$ and we denote this group by $\underline{G}(v)$. When the degree of each vertex is bigger than $6$, those groups of automorphisms of trees have been extensively studied and classified by Radu in \cite{Radu2017}. For that reason we call them \textbf{Radu groups}. It is not hard to realize that those groups are type I. Indeed, each Radu group $G$ contains a cocompact subgroup $H$ that is conjugate in $\Aut(T)$ to the semi-regular version of the universal group of Burger-Mozes $\Alt_{(i)}(T)^+$ see \cite[page 4.]{Radu2017}. Since $H$ is both open and cocompact in $G$, \cite[Theorem 1]{KallmanI1973} ensures that $G$ is type I if and only if $H$ is type I. On the other hand, when the degree of each vertex is bigger than $4$, $\Alt_{(i)}(T)^+$ acts transitively on the boundary and satisfies the Tits independence property. It follows from \cite{Amann2003} and \cite{Ciobotaru2015} that $H$ is a type I group which proves that every Radu group is Type I. The purpose of these notes is to go further. Inspired by Ol’shanskii’s work and the recent progress achieved in the abstraction of his framework \cite{Semal2021O}, we give a classification of the irreducible representations of simple Radu groups and deduce a description of the irreducible representations of any Radu groups. Among other things, this provides the following contribution to Nebbia's CCR conjecture on trees.  
	\begin{theoremletter}\label{Theorem A}
		Let $T$ be a $(d_0,d_1)$-semi-regular tree with $d_0,d_1\geq 6$. Then, Radu groups are uniformly admissible and hence CCR.
	\end{theoremletter}
	To put this result into the perspective of Radu's paper, we recall that the local action $\underline{G}(v)\leq \Sym(E(v))$ at every vertex $v\in T$ of a closed subgroup $G\leq \Aut(T)$ that is $2$-transitive on the boundary is a $2$-transitive subgroup of $\Sym(E(v))$ \cite[Lemma 3.1.1]{BurgerMozes2000}. On the other hand, \cite[Proposition B.1 and Corollary B.2]{Radu2017} ensure that
	\begin{equation*}
	\Theta= \{d \geq 6 \lvert \mbox{ each finite }\mbox{2-transitive subgroup of }\Sym(d)\mbox{ contains Alt}(d)\}
	\end{equation*}
	is asymptotically dense in $\N$ and its ten smallest elements are $34$,
	$35$, $39$, $45$, $46$, $51$, $52$, $55$, $56$ and $58$. All together, this implies the following.
	\begin{theoremletter}
		Nebbia's CCR conjecture on trees is confirmed for any $(d_0,d_1)$-semi-regular tree with $d_0,d_1\in \Theta$ where $\Theta$ is the asymptotically dense subset of $\N$ defined above.
	\end{theoremletter}
	We now explain how we obtained a classification of the irreducible representations of simple Radu groups. We first recall that the irreducible representations of a closed automorphism group $G\leq \Aut(T)$ of a locally finite tree $T$ splits in three categories. An irreducible representation $\pi$ of $G$ is called:
	\begin{itemize}\label{definition de spheric special cuspidal}
		\item \tg{spherical} if there exists a vertex $v\in V(T)$ such that $\pi$ admits a non-zero $\Fix_G(v)$-invariant vector where $\Fix_G(v)=\{g\in G\lvert gv=v\}$.
		\item \tg{special} if it is not spherical and there exists an edge $e\in E(T)$ such that $\pi$ admits a non-zero $\Fix_G(e)$-invariant vector where $\Fix_G(e)=\{g\in G\lvert gv=v\qq\forall v\in e\}$ is the fixator of the edge $e$.
		\item \tg{cuspidal} if it is neither spherical nor special.
	\end{itemize}
	The spherical and special representations are classified since the end of the 70's at the level of generality of the conjecture that is for any closed non-compact subgroup $G\leq \Aut(T)$ acting transitively on the boundary of the tree see \cite{Matsumoto}, \cite{Olshanskii1977} and \cite{Olshanskii1980}. Furthermore, we recall that Matsumoto's work emphasises a strong connection between those kinds of representations and the irreducible representations of Hecke algebras. To be more precise, we recall that a group acting $2$-transitively on the boundary is either type-preserving or admits an index $2$ closed type-preserving subgroup acting $2$-transitively on the boundary. Since \cite[Corollary 3.6.]{CapraceCiobotaru2015} ensures that every such group $G$ comes from a B-N pair, each spherical or special representation of $G$ defines an irreducible representation of the associated Hecke algebra $C_c(B\backslash G\slash B)$ of continuous compactly supported $B$-bi-invariant functions $f:G\rightarrow \C$ where $B=\Fix_G(e)$ is the pointwise fixator of an edge $e\in E(T)$. Matsumoto's works enlightened the fact that this correspondence is actually bijective see {\cite[Chapter 5, Section 6]{Matsumoto}}.
	
	The cuspidal representations on the other hand are not classified at the level of generality of the conjecture. Nevertheless, a complete classification of those representations was achieved for certain families of groups. Among non-linear groups for instance, Ol’shanskii's work lead to a classification of the cuspidal representations for any closed group of automorphisms of a semi-regular tree satisfying the Tits independence property see \cite{Olshanskii1977}, \cite{Amann2003}. 
	The main idea leading to this classification was to exploit the independence of the action on the tree to obtain a particular factorization on a well chosen basis of neighbourhood of the identity made by compact open subgroups. When it comes to Radu groups, \cite{Radu2017} highlighted the fact that those groups are defined by local conditions. Among other things, when $T$ is a $(d_0,d_1)$-semi regular tree with $d_0,d_1\geq 4$, Radu introduced a family of groups $G_{(i)}(Y_0,Y_1)$ indexed by two finite subsets $Y_0,Y_1\subseteq \N$ see Definition \ref{definition de GiXY} below. Furthermore, he showed that those groups are abstractly simple and that they exhaust the list of simple Radu groups when $d_0,d_1\geq 6$. Our paper takes advantage of the recent abstraction of Ol'shanskii's framework developed in \cite{Semal2021O} and the description of those groups provided by Radu to obtain a classification their cuspidal representations see Section \ref{Section classification des rep de simple radu groups}. The author would like to underline that an application of Ol'shanskii's machinery on Radu groups already exists in \cite[Section 4]{Semal2021O} since they satisfy a generalisation of the Tits independence property (the property ${\rm IP} _k$ defined in \cite{BanksElderWillis2015}). However, unless the property ${\rm IP} _k$ coincides with the Tits independence property ($k=1$), this approach never leads to a classification of all cuspidal representation of the group. However, the approach considered in the present paper relies on the independence provided by local conditions rather than the property ${\rm IP}_k$. In particular, by contrast with the approach developed in \cite{Semal2021O}, Section \ref{Section classification des rep de simple radu groups} below leads to a description of every cuspidal representation of the groups $G_{(i)}(Y_0,Y_1)$. 
	\begin{theoremletter}\label{theorem C}
		In a $(d_0,d_1)$-semi-regular tree with $d_0,d_1\geq 4$, the cuspidal representations of $G_{(i)}(Y_0,Y_1)$ are in bijective correspondence given by induction with a family of irreducible representations of compact open subgroups. This correspondence is explicitly described by $\lb$Theorem \ref{la classification des reresentations cuspidale}, Section \ref{Section classification des rep de simple radu groups}$\rb$.
	\end{theoremletter}
	Among other things, this proves the cuspidal representations are induced from compact open subgroups and therefore square-integrable. Since \cite[Corollary of Theorem 2]{Harish1970} ensures, for every compact open subgroup $K\leq G$, the existence a positive integer $k_K$ such that $\dim(\Hr{\pi}^K)\leq k_K$ for all square-integrable representations $\pi$ of $G$ and as a consequence of the classification of the spherical representations (see Section \ref{Section spherical and special rep}) this leads to the conclusion that the groups $G_{(i)}(Y_0,Y_1)$ are uniformly admissible (see Section \ref{section radu groups are Type I}). On the other hand, when $d_0,d_1\geq 6$ Radu's classification ensure that every Radu group $G$ belongs to a finite chain $H_n\geq... \geq H_0$ with $n\in \{0,1,2,3\}$ such that $H_n=G$, $\lb H_t: H_{t-1} \rb =2$ for all $t$ and $H_0$ is conjugate in the group of type-preserving automorphisms $\Aut(T)^+$ to one of those $G_{(i)}^+(Y_0,Y_1)$. Since Mackey's machinery allows one to describe the irreducible representation of a locally compact $G$ in terms of the irreducible representations of any of its closed subgroup $H$ of index $2$ (see Appendix \ref{section irreducvtible pour des groupes d'indice 2}), this leads to a description of the cuspidal representations of any other Radu group. Furthermore, this also shows that every other Radu group is uniformly admissible see Lemma \ref{lemma G admissible iff  H is admissible}.

	\subsection*{Structure of the paper}
	In Section \ref{Section spherical and special rep} we recall the classification of spherical and special representations of any closed non-compact subgroup $G\leq \Aut(T)$ acting transitively one the boundary see \cite{Matsumoto}, \cite{Olshanskii1977} and \cite{Olshanskii1980}. The purpose of Section \ref{section Radu groups} is to recall Radu's classification of Radu groups \cite{Radu2017} and the definition of the $G_{(i)}(Y_0,Y_1)$. In Section \ref{Section classification des rep de simple radu groups}, we recall the notion of \textbf{Ol'shanskii's factorization} developed in \cite{Semal2021O} and obtain a classification of the cuspidal representations of the $G_{(i)}(Y_0,Y_1)$. The complete classification of the irreducible representations of $G_{(i)}(Y_0,Y_1)$ resulting from Sections \ref{Section spherical and special rep} and \ref{Section classification des rep de simple radu groups} is then used in Section \ref{section radu groups are Type I} to prove  uniform admissibility. Finally, the purpose of Appendix \ref{section irreducvtible pour des groupes d'indice 2} is to recall the procedure allowing one to describe the irreducible representation of a locally compact $G$ in terms of the irreducible representations of any of its closed subgroup $H$ of index $2$. In particular, this appendix provides a way to obtain the irreducible representations of any Radu groups from the irreducible representations of the abstractly simple Radu groups $G_{(i)}(Y_0,Y_1)$ and shows that other Radu groups are also uniformly admissible. 
	\subsection*{Acknowledgements}
	I warmly thank Pierre-Emmanuel Caprace for all the insightful discussions we shared and for his comments on preliminary versions of this paper. 
	\section{Spherical and special representations}\label{Section spherical and special rep}
	Let $T$ be a $(d_0,d_1)$-semi-regular tree with $d_0,d_1\geq 3$. We recall that a tree $T$ is called $(d_0,d_1)$-\tg{semi-regular} if there exists a bipartition $V(T)=V_0\sqcup V_1$ of $T$ such that each vertex of $V_i$ has degree $d_i$ and every edge of $T$ contains exactly one vertex in each $V_i$. As explained in the introduction, the irreducible representations of any closed subgroup $G\leq \Aut(T)$ of the group of automorphisms of such a tree splits in three categories. Those representations are either  \tg{spherical}, \tg{special} or \tg{cuspidal}. The purpose of the present section is to recall the classification of spherical and special representations of any closed non-compact subgroup $G\leq \Aut(T)$ acting transitively on the boundary (this covers every Radu group). This classification is a classical result known since the end of the 70's and we claim no originality. Furthermore, we refer to \cite{Matsumoto}, \cite{Olshanskii1977} and \cite{FigaNebbia1991} for details. 
	
	The details of this classification are gathered in Theorems \ref{thm la classification des spherique cas trans}, \ref{thm la classification des spherique cas 2 orbites}, \ref{thm la classification des speciale} below. We now recall preliminaries required for the statement of those results. Given a locally compact group $G$ and a compact subgroup $K\leq G$, we say that $(G, K)$ is a \textbf{Gelfand pair} if the convolution algebra $C_c(K \backslash  G \slash K)$ of compactly supported, continuous $K$-bi-invariant functions on $G$ is commutative. Now, let $(G,K)$ be a Gelfand pair and let $\mu$ be the left-Haar measure of $G$ renormalised in such a way that $\mu(K)=1$. A function $\varphi: G\rightarrow \C$ is called $K$\textbf{-spherical} if it is a $K$-bi-invariant continuous function with $\varphi(1_G)=1$ and $$\int_{K}\varphi(gkg')\qq \diff \mu(k)=\varphi(g)\varphi(g') \q \forall g,g'\in G.$$ 
	The interests of those notions lay in the following result.
	\begin{theorem}[{\cite[Chapter IV. \textsection{3}, Theorem $3$ and $9$]{Serge1985}}]
		Let $(G,K)$ be a Gelfand pair. For every irreducible representation $\pi$ of $G$ we have that $\dim(\Hr{\pi}^K)\leq 1$. Furthermore, there is a bijective correspondence $\pi \rightarrow \varphi_\pi$ with inverse map given by the GNS construction between the equivalence classes of irreducible representations of $G$ with non-zero $K$-invariant vectors and the $K$-spherical functions of positive type on $G$ (the function $\varphi_\pi$ is the function $\varphi_{\pi}(g)=\prods{\pi(g)\xi}{\xi}$ corresponding to any unit vector $\xi \in \Hr{\pi}^K$).
	\end{theorem}
	We are finally ready to recall the details of the classification of spherical and special representations for any non-compact closed subgroups $G\leq\Aut(T)$ acting transitively on the boundary of $T$. We recall that those groups act transitively on the edges of $T$ and have therefore either one or two orbits of vertices. We treat those cases separately. 

	\begin{theorem}[{\cite[Chapter II]{Nebbia1999}}]\label{thm la classification des spherique cas trans}
		Let $T$ be a $d$-regular tree, let $v\in V(T)$ and let $G\leq \Aut(T)$ be a closed non-compact subgroup acting transitively on the vertices of $T$ and the boundary $\partial T$. Then, $(G, \Fix_G(v))$ is a Gelfand pair and every spherical representation of $G$ admits a non-zero $\Fix_G(v)$-invariant vector. Furthermore, the equivalence classes of spherical representations of $G$ are in bijective correspondence with the interval $\lb -1; 1\rb$ via the map $\phi_v:\pi\mapsto \varphi_\pi(\tau_v)$ where $\tau_v$ is any element of $G$ such that $d(\tau_v v,v)=1$ and $\varphi_\pi$ is the unique $\Fix_G(v)$-spherical function of positive type attached $\pi$. Under this correspondence, the trivial representation corresponds to $1$.
	\end{theorem}
	The following theorem is obtained from \cite{Matsumoto} but is formulated differently for coherence of our expository. 
	\begin{theorem}[{\cite[Chapter 5, Section 6]{Matsumoto}}]\label{thm la classification des spherique cas 2 orbites}
		Let $T$ be a $(d_0,d_1)$-semi-regular tree with $d_0,d_1\geq 3$, let $v\in V(T)$, let $v'$ be any vertex at distance one from $v$ and let $G\leq \Aut(T)$ be a closed non-compact subgroup of type-preserving automorphisms acting transitively on the boundary $\partial T$. Then, there is exactly one spherical representation $\pi_v$ of $G$ with a non-zero $\Fix_G(v)$-invariant vector but no non-zero $\Fix_G(v')$-invariant vector. Furthermore, $(G, \Fix_G(v))$ is a Gelfand pair and apart from the two exceptional representations $\pi_v$ and $\pi_{v'}$, every spherical representation of $G$ admits, for all $w\in V(T)$, a non-zero $\Fix_G(w)$-invariant vector. In addition, if $v'$ has degree $d'$, the equivalence classes of spherical representations admitting a non-zero $\Fix_G(v)$-invariant vector are in bijective correspondence with the interval $\big\lb -\frac{1}{d'-1};1\big\rb$ via the map $\phi_v:\pi \mapsto \varphi_\pi(\tau_v)$ where $\tau_v$ is an element of $G$ such that $d(\tau_v v,v)=2$. Under this correspondence, the exceptional spherical representation $\pi_v$ corresponds to $-\frac{1}{d'-1}$ and the trivial representation corresponds to $1$. Finally, if $\pi$ is a non-exceptional spherical representation of $G$ we have that 
		$$\phi_{v'}(\pi)=\frac{d(d'-1)}{d'(d-1)}\phi_v(\pi) + \frac{d-d'}{d'(d-1)}.$$ 
	\end{theorem}
	To describe the special representations, let $T$ be a $(d_0,d_1)$-semi-regular tree with $d_0,d_1\geq 3$, let $e\in E(T)$ and let $G\leq \Aut(T)$ be a closed subgroup acting transitively on the edges of $T$. We define $\mathcal{L}(e)$ as the subspace of $\Fix_{G}(e)$-right invariant square-integrable functions $\varphi : G \rightarrow \C$ satisfying $$\int_{\Fix_G(v)}\varphi(gk)\qq \diff \mu(k)=0\q \forall g\in G, \forall v\in e.$$ 
	Notice that $\mathcal{L}(e)$ is a closed left invariant subspace of $L^2(G)$ and let $\sigma:G\rightarrow \mathcal{U}(\mathcal{L}(e))$ be the unitary representation of $G$ defined by $\sigma(t)\varphi(g)=\varphi(t^{-1}g)$ $\forall g,t\in G$, $\forall \varphi\in \mathcal{L}(e)$. If $G$ is transitive on the vertices of $T$, we choose an  inversion  $h\in G$ of the edge $e$ and consider the map $\nu:\mathcal{L}(e) \rightarrow \mathcal{L}(e)$ defined by $\nu(\varphi)(g)=\varphi(gh)$ $\forall \varphi\in \mathcal{L}(e)$, $\forall g\in G$. This map is well defined since for all $\varphi\in \mathcal{L}(e)$, for all $g\in G$ and every $v\in e$ we have 
	\begin{equation*}
	\begin{split}
	\int_{\Fix_G(v)} (\nu\varphi)(gk)\diff \mu(k)&=\int_{\Fix_G(v)} \varphi(gkh)\diff \mu(k)\\
	&=\int_{\Fix_G(v)} \varphi(ghh^{-1}kh)\diff \mu(k)\\
	&=\int_{\Fix_G(h^{-1}v)} \varphi(ghk)\diff \mu(k)=0.
	\end{split}
	\end{equation*} 
	On the other hand, since every element of $\mathcal{L}(e)$ is $\Fix_G(e)$-right invariant, notice that $\nu$ is an involution and that it does not depend on our choice of inversion of the edge $e$. For every $\epsilon\in \{-1,1\}$, we let $\mathcal{L}(e)_{\epsilon}$ be the eigenspace of $\nu$ and we $\sigma^\epsilon:G\rightarrow \mathcal{U}(\mathcal{L}(e)_{\epsilon})$ be the unitary representation of $G$ defined by $\sigma^\epsilon(t)\varphi(g)=\varphi(t^{-1}g)$ $\forall g,t\in G$, $\forall \varphi\in \mathcal{L}(e)_{\epsilon}$. We are now ready to state the classification of special representations.
	\begin{theorem}[{\cite[Chapter III, Section 2]{FigaNebbia1991}}, {\cite[Section 5.6]{Matsumoto}}]\label{thm la classification des speciale}
		Let $T$ be a $(d_0,d_1)$-semi-regular tree with $d_0,d_1\geq 3$, let $e\in E(T)$ and let $G\leq \Aut(T)$ be a closed non-compact subgroup acting transitively on the boundary $\partial T$. Every special representation of $G$ is square integrable and admits a $\Fix_G(f)$-invariant vector for every $f\in E(T)$. Furthermore:
		\begin{enumerate}
			\item If $G$ acts transitively on $V(T)$, $(\sigma^{-1},\mathcal{L}(e)_{-1})$ and $(\sigma^{+1},\mathcal{L}(e)_{+1})$ are representatives of the two equivalence classes of special representations.
			\item If $G$ has two orbits on $V(T)$, $(\sigma,\mathcal{L}(e))$ is a representative of the unique equivalence class of special representations.
		\end{enumerate}
	\end{theorem}
	
	\section{The Radu groups}\label{section Radu groups}
	Let $T$ be a $(d_0,d_1)$ semi-regular tree with $d_0,d_1\geq 4$ and bipartition $V(T)=V_0\sqcup V_1$ and let $\Aut(T)^+$ denote the group of type-preserving automorphisms of $T$ that is the set of automorphisms of $T$ which leave $V_0$ and $V_1$ invariants. The purpose of this section is to recall the classification of Radu groups \cite{Radu2017}. To this end, we set 
	\begin{equation*}
	\mathcal{H}_T=\{ G\leq\Aut(T)\mid G\mbox{ is closed and } 2\mbox{-transitive on }\partial T \}
	\end{equation*}
	and 
	\begin{equation*}
	\mathcal{H}_T^+=\{ G\leq\Aut(T)^+\mid G\mbox{ is closed and } 2\mbox{-transitive on }\partial T \}.
	\end{equation*} 
	If $d_0 \not= d_1$, notice that every automorphisms of $T$ is type-preserving so that $\mathcal{H}_T^+= \mathcal{H}_T$. We recall that for each vertex $v\in V(T)$, the stabilizer $\Fix_G(v)$ of $v$ acts on the set $E(v)$ of edges containing $v$ and that the image of $\Fix_G(v)$ in $\Sym(E(v))$ for this projection map (which  we denote by $\underline{G}(v)$) is called the \textbf{local action} of $G$ at $v$. Furthermore, we recall in the light of \cite[Lemma 3.1.1]{BurgerMozes2000}, that every group $G\in \mathcal{H}_T^+$ is transitive on $V_0$ and $V_1$. Hence, all the groups $\underline{G}(v)$ with $v\in V_0$ (respectively $v\in V_1$) are permutation isomorphic to the same group $F_0\leq \Sym(d_0)$ (respectively $F_1\leq \Sym(d_1)$). We recall that the groups $G\in \mathcal{H}_T$ such that $\underline{G}(v)\cong F_t\geq \Alt(d_t)$ for every vertex $v\in V_t(T)$ and for $t\in \{0,1\}$ are called \textbf{Radu groups}. Those groups have been extensively studied and classified by N.~Radu in \cite{Radu2017} when $d_0,d_1\geq 6$. The purpose of this section is to recall his classification.
	
	We start by recalling a few definitions needed to describe Radu groups. For every vertex $v\in V(T)$ and every positive integer $r\in \N$ let $$S(v,r)=\{w\in V(T)\mid d(v,w)=r\}$$ be the set of vertices of $T$ at distance $r$ from $v$.
	\begin{definition}
		A \tg{legal coloring} $i:V(T)\rightarrow \N$ of $T$ is the concatenation of a pair of maps $$i_0:V_0\rightarrow\{1,...,d_1\}\mbox{ and }i_1:V_1\rightarrow\{1,...,d_0\}$$ 
		such that $\restr{i_0}{S(v,1)}:S(v,1)\rightarrow \{1,...,d_1\}\mbox{ and }\restr{i_1}{S(w,1)}:S(w,1)\rightarrow \{1,...,d_0\}$ are bijections for all $v\in V_1$  and $w\in V_0$.
	\end{definition} 
	Given a legal coloring $i$ of $T$ and an automorphism $g\in \Aut(T)$, the \tg{local action} of $g$ at a vertex $v\in V(T)$ is defined as the following permutation: 
	\begin{equation*}
	\sigma_{(i)}(g,v)= \restr{i}{S(gv,1)}\circ g\circ \Big(\restr{i}{S(v,1)}\Big)^{-1}\in \begin{cases}
	\Sym(d_0)\q &\mbox{if }v\in V_0\\
	\Sym(d_1)\q &\mbox{if }v\in V_1
	\end{cases}.
	\end{equation*}  
	\begin{remark}
		If $d_0=d_1$, the tree $T$ is a regular tree and this notion of legal coloring and local action of an element differ from the notion of legal coloring and local action used to define the universals Burger Mozes groups in \cite{BurgerMozes2000}. Indeed, with our definition, the closed subgroup $G\leq \Aut(T)$ of all automorphisms of trees $g\in G$ such that $\sigma_{(i)}(g,v)=\id$ $\forall v\in V$ is not transitive on the set of vertices of $T$ (not even transitive on $V_0$).
	\end{remark}
	
	Now, let $T$ be a $(d_0,d_1)$-semi-regular tree with $d_0,d_1\geq 4$ and let $i$ be a legal coloring of $T$. For every vertex $v\in V(T)$ and every finite set $Y\subseteq \N$ let $$S_Y(v)=\bigcup_{r\in Y} S(v,r)$$ and for every set of vertices $B\subseteq V(T)$ let   
	$$\Sgn_{(i)}(g,B)=\prod_{w\in B}^{}\sgn(\sigma_{(i)}(g,w))$$
	where $\sgn(\sigma_{(i)}(g,w))$ is the sign of the local action $\sigma_{(i)}(g,w)$ of the automorphism $g$ at $w$ for the legal coloring $i$. 
	The following groups will have a central importance in the rest of this paper.
	\begin{definition}\label{definition de GiXY}
		For all (possibly empty) finite sets $Y_0, Y_1$ of $\N$ and every legal coloring $i$ of $T$, we set
		\begin{equation*}
		G^+_{(i)}(Y_0,Y_1)=\left\{g\in \Aut(T)^+\ \middle\vert \begin{array}{l}
		\Sgn_{(i)}(g,S_{Y_0}(v))=1\qq \mbox{ for each }v\in V_{t_0},\\ 
		\Sgn_{(i)}(g,S_{Y_1}(v))=1\qq \mbox{ for each }v\in V_{t_1}
		\end{array}\right\},
		\end{equation*}
		where $t_0=\max(Y_0)\mod2$,  $t_1=(1+\max(Y_1))\mod2$ and $\max(\es)=0$. 
	\end{definition}
	\begin{remark}\label{remark les sommets ont des types opposers}
		Notice, that the choices of $t_0$ and $t_1$ are made in such a way that the vertices of $S_{Y_0}(v)$ with $v\in V_{t_0}$ at maximal distance from $v$ and the vertices of $S_{Y_1}(w)$ with $w\in V_{t_1}$ at maximal distance from $w$ have opposite types.
	\end{remark}
	Notice that $G_{(i)}^+(\es,\es)=\Aut(T)^+$ is the full group of type-preserving automorphisms and that $G^+_{(i)}(\{0\},\{0\})$ is a subgroup of each $G^+_{(i)}(Y_0,Y_1)$. Furthermore, if $T$ is a $d$-regular tree notice that $G^+_{(i)}(\{0\},\{0\})$ is conjugate to $U(\Alt(d))^+$ where $G^+=G\cap \Aut(T)^+$ and $U(\Alt(d))$ is the universal Burger-Mozes group of the alternating group see \cite{BurgerMozes2000}.
	
	As we recall below, when $d_0,d_1\geq 6$, every abstractly simple Radu group is of the form $G_{(i)}^+(Y_0,Y_1)$ for some finite $Y_0,Y_1\subseteq \N$ and some legal coloring $i$ of $T$. Furthermore, when $d_0,d_1\geq 6$, every Radu group $G$ belongs to a finite chain $H_n\geq... \geq H_0$ with $n\in \{0,1,2,3\}$ such that $H_n=G$, $\lb H_t: H_{t-1} \rb =2$ for all $t$ and $H_0$ is conjugate in $\Aut(T)^+$ to one of those $G_{(i)}^+(Y_0,Y_1)$. In particular, using the Appendix \ref{section irreducvtible pour des groupes d'indice 2} the irreducible representations of every Radu groups can be obtained from the irreducible representations of the $G_{(i)}^+(Y_0,Y_1)$ and vice versa. We now recall more precisely the statements proved in \cite{Radu2017} that will be used in this paper.
	\begin{theorem*}[{\cite[Theorem A]{Radu2017}}]
		Let $T$ be a $(d_0,d_1)$-semi-regular tree with $d_0,d_1\geq 4$ and let $i$ be a legal coloring of $T$. Then, for every finite subsets $Y_0, Y_1\subseteq \N$ the group $G^+_{(i)}(Y_0,Y_1)$ belongs to $\mathcal{H}^+_T$ and is abstractly simple.
	\end{theorem*}
	The following results ensure that every Radu group contains a conjugate of such a $G^+_{(i)}(Y_0,Y_1)$ with finite index. To formulate this precisely, we introduce some notations. For every locally compact group $G$ we let $G^{(\infty)}$ be the intersection of all normal cocompact closed subgroups of $G$. We recall that for any group $H\in \mathcal{H}^+_T$, Burger and Mozes proved that $H^{(\infty)}$ belongs to $\mathcal{H}^+_T$ and is topologically simple \cite[Proposition 3.1.2]{BurgerMozes2000} (in our cases, it is even abstractly simple). Finally, we let $\mathcal{G}_{T}^+(i)$ be the set of groups $G^+_{i}(Y_0,Y_1)$ with non-empty finite $Y_0, Y_1\subseteq\N$ such that $y=\max(Y_t)\mod2$ for each $y\in Y_t$ with $y\geq \max(Y_{1-t})$ ($t\in \{0,1\}$). 
	\begin{theorem*}[{\cite[Theorem B]{Radu2017}}]
		Let $T$ be a $(d_0,d_1)$-semi-regular tree with $d_0,d_1\geq 6$, let $i$ be a legal coloring of $T$ and let $G\in \mathcal{H}^+_T$ be such that $\underline{G}(v)\cong F_0\geq \Alt(d_0)$ for each $v\in V_0$ and $\underline{G}(w)\cong F_1\geq \Alt(d_1)$ for each $w\in V_1$. Then, we have $\lb G: G^{(\infty)}\rb \in \{1,2,4\}$ and $G^{(\infty)}$ is conjugate in $\Aut(T)^+$ to an element of $\mathcal{G}_{T}^+(i)$.
	\end{theorem*}
	When $T$ is a $d$-regular tree, a similar result holds for all $G\in \mathcal{H}_T- \mathcal{H}_T^+$.
	\begin{theorem*}[{\cite[Corollary C]{Radu2017}}]
		Let $T$ be a $d$-regular tree with $d\geq 6$ and let $i$ be a legal coloring of $T$ and let $G\in \mathcal{H}_{T}- \mathcal{H}^+_T$ be such that $\underline{G}(v)\cong F\geq\Alt(d)$ for each $v\in V(T)$. Then, we have $\lb G: G^{(\infty)} \rb\in \{2,4,8\} $ and $G^{(\infty)}$ is conjugate to ${G}^+_{(i)}(Y,Y)$ for some finite subset $Y$ of $\N$.
	\end{theorem*}
	The following theorem follows from Radu's description of Radu groups.
	\begin{theorem}\label{Corollary Radu simple then Radu}
		Let $T$ be a $(d_0,d_1)$-semi-regular tree with $d_0,d_1\geq 6$ and let $i$ be a legal coloring of $T$. Every Radu group $G$ admits a finite chain $H_n\geq... \geq H_0$ with $n\in \{0,1,2,3\}$ such that $H_n=G$, $\lb H_t: H_{t-1} \rb =2$ for all $t$ and $H_0$ is conjugate in $\Aut(T)^+$ to $G_{(i)}^+(Y_0,Y_1)$. 
	\end{theorem}

	\section{Cuspidal representations of the simple Radu groups}\label{Section classification des rep de simple radu groups}
	Let $T$ be a $(d_0,d_1)$-semi-regular tree with $d_0,d_1\geq 4$ and let $V(T)=V_0\sqcup V_1$ be the associated bipartition. Let $i$ be a legal coloring of $T$ and let $Y_0,Y_1\subseteq \N$ be two finite subsets. Recall from Section \ref{section Radu groups} that $G_{(i)}^+(Y_0,Y_1)$ (Definition \ref{definition de GiXY}) is a closed abstractly simple subgroups of $\Aut(T)^+$ acting $2$-transitively on the boundary $\partial T$ and whose local action at every vertex contains the alternating group. Furthermore, when $d_0,d_1 \geq 6$, Radu's classification ensures that every simple Radu group is of this form. Our current purpose is to describe the irreducible representations of $G_{(i)}^+(Y_0,Y_1)$ and show that this group is uniformly admissible, hence CCR.

	We recall from the introduction that the irreducible representations of $G_{(i)}^+(Y_0,Y_1)$ splits in three categories. Those are either \tg{spherical}, \tg{special} or \tg{cuspidal}. A classification of the spherical and special representations of any subgroup $G\leq \Aut(T)$ acting $2$-transitively on the boundary is already given in Section \ref{Section spherical and special rep}. In particular, this classification applies to the spherical and special representations of $G_{(i)}^+(Y_0,Y_1)$. Our current purpose is to give a description of the cuspidal representations of those groups. As announced in the introduction, our idea is to take advantage of the recent abstraction of Ol'shanskii's framework developed in \cite{Semal2021O} and the description of those groups provided by Radu. The main concept developed in \cite{Semal2021O}, is the concept of Ol'shanskii's factorization see Definition \ref{definition olsh facto} below. We recall that such a factorization leads to a description of the irreducible representations admitting particular invariant vectors as induced representations from compact open subgroups. 
	The author would like to recall that an Ol'shanskii's factorization for Radu groups is already provided by \cite[Section 4]{Semal2021O} since they satisfy a generalisation of the Tits independence property (the property ${\rm IP} _k$ defined in \cite{BanksElderWillis2015}). However, this approach never leads to a description of every cuspidal representations of $G_{(i)}^+(Y_0,Y_1)$ unless the property property ${\rm IP}_k$ coincides with the Tits independence property (that is $k=1$). By contrast, the approach developed in the present section relies on the independence provided by local conditions given by Definition \ref{definition de GiXY} rather than the property ${\rm IP}_k$ and a description of every cuspidal representation of $G_{(i)}^+(Y_0,Y_1)$ is obtained in Section \ref{section cuspidal rep description} below.
	
	\subsection{Preliminaries}
	The purpose of this section is to recall the axiomatic framework developed in \cite{Semal2021O} (we refer to this paper for details). This machinery will then be used in the following sections to obtain a description of the cuspidal representations of the Radu groups $G_{(i)}^+(Y_0,Y_1)$.
	
	Let $G$ be a totally disconnected locally compact group, let $\mathcal{B}$ denote the set of compact open subgroups of $G$, $P(\mathcal{B})$ denote the power set of $\mathcal{B}$ and let
	$$\mathcal{C}: \mathcal{B}\rightarrow P(\mathcal{B})$$
	be the map sending a compact open subgroup to its conjugacy class in $G$. Let $\mathcal{S}$ be a basis of neighbourhoods of the identity consisting of compact open subgroups of $G$ and let $\mathcal{F}_{\mathcal{S}}=\{ \mathcal{C}(U)\lvert U\in \mathcal{S}\}$. We equip $\mathcal{F}_\mathcal{S}$ with the partial order given by the reverse inclusion of representatives ($\mathcal{C}(U)\leq \mathcal{C}(V)$ if there exists $\tilde{U}\in \mathcal{C}(U)$ and $\tilde{V}\in \mathcal{C}(V)$ such that $\tilde{V}\subseteq \tilde{U}$). For a poset $(P, \leq)$ and an element $x\in P$, we recall that the \tg{height} of $x$ in $(P,\leq)$ is $L_x-1$ where $L_x$ is the maximal length of a strictly increasing chain in $P_{\leq x}=\{y\in P\lvert y\leq x\}$ if such a maximal length exists and we say that the height is infinite otherwise. 
	\begin{definition}
		A basis of neighbourhoods of the identity $\mathcal{S}$ consisting of compact open subgroups of $G$ is called a \tg{generic filtration} of $G$ if the height of every element in $\mathcal{F}_{\mathcal{S}}$ is finite.
	\end{definition} 
	
	Every generic filtration $\mathcal{S}$ of $G$ splits as a disjoint union $\mathcal{S}=\bigsqcup_{l\in \N}\mathcal{S}\lb l \rb$ where $\mathcal{S}\lb l\rb$ denotes the set of elements $U\in \mathcal{S}$ such that $\mathcal{C}(U)$ has height $l$ in $\mathcal{F}_{\mathcal{S}}$. The element of $\mathcal{S}\lb l \rb$ are called the elements at \tg{depth} $l$. Since $\mathcal{S}$ is a basis of neighbourhood of the identity consisting of compact open subgroups of $G$, notice that for every irreducible representation $\pi$ of $G$, there exists a group $U\in \mathcal{S}$ such that $\pi$ admits non-zero $U$-invariant vector. In particular, for every irreducible representation $\pi$ of $G$ there exists a smallest non-negative integer $l_\pi\in \N$ such that $\pi$ admits non-zero $U$-invariant vectors for some $U\in \mathcal{S}\lb l_\pi \rb$. This $l_\pi$ is called the \tg{depth} of $\pi$ with respect to $\mathcal{S}$.
	
	The key notion developed in \cite{Semal2021O} is the notion of \textbf{factorization} at depth $l$ for a generic filtration $\mathcal{S}$ that we now recall.
	\begin{definition}\label{definition olsh facto}
		Let $G$ be a non-discrete unimodular totally disconnected locally compact group, let $\mathcal{S}$ be a generic filtration of $G$ and let $l$ be a strictly positive integer. We say that $\mathcal{S}$ \tg{factorizes at depth} $l$ if the following conditions hold:
		\begin{enumerate}
			\item For all $U\in\mathcal{S}\lb l \rb$ and every $V$ in the conjugacy class of an element of $\mathcal{S}$ such that $V \not\subseteq U$, there exists $W$ in the conjugacy class of an element of $\mathcal{S}\lb l -1\rb$ such that: $$U\subseteq W \subseteq V U=\{vu\lvert u\in U, v\in V\}.$$
			\item For all $U\in\mathcal{S}\lb l \rb$ and every $V$ in the conjugacy class of an element of $\mathcal{S}$, the set
			\begin{equation*}
			N_G(U, V)= \{g\in G \lvert g^{-1}Vg\subseteq U\}
			\end{equation*}
			is compact.
		\end{enumerate} 
		Furthermore, the generic filtration $\mathcal{S}$ of $G$ is said to \tg{factorize$^+$ at depth} $l$ if in addition for all $U\in\mathcal{S}\lb l\rb$ and every $W$ in the conjugacy class of an element of $\mathcal{S}\lb l-1\rb$ such that $U\subseteq W$ we have
		\begin{equation*}
		W\subseteq N_G(U,U) =\{g\in G \mid g^{-1}Ug\subseteq U\}.
		\end{equation*} 
		Since $G$ is unimodular, notice that the set $N_G(U,U)$ coincides with the normalizer $N_G(U)$ of $U$ in $G$. 
	\end{definition} 
	The relevance of this notion is given by \cite[Theorem A]{Semal2021O} which lead to a description of the irreducible representations at height $l$ in terms of a family of irreducible representations of finite groups called $\mathcal{S}$-standard representations (Definition \ref{definitoin standard representations} below) if the generic filtration factorizes$^+$ at height $l$.
	
	\subsection{Generic filtration for $G^+_{(i)}(Y_0,Y_1)$}\label{chapter generic filtration for Gi+YY}
	
	Let $T$ be a $(d_0,d_1)$-semi-regular tree with $d_0,d_1\geq 4$ and let $V(T)=V_0\sqcup V_1$ be the associated bipartition. Let $i$ be a legal coloring of $T$ and let $Y_0,Y_1\subseteq \N$ be two finite subsets. The purpose of this section is to explicit a generic filtration for $G_{(i)}^+(Y_0,Y_1)$. To this end, let $\mathfrak{T}_0$ be the family of subtrees of $T$ defined by
	$$\mathfrak{T}_{0}=\{B_T(v,r)\lvert v\in V(T), r\geq 1\}\sqcup\{B_T(e,r)\lvert e\in E(T), r\geq 0\}$$
	and consider the basis of neighbourhoods of the identity given by the fixators of those trees
	\begin{equation*}
	\mathcal{S}_{0}=\{\Fix_G(\mathcal{T})\lvert \mathcal{T}\in \mathfrak{T}_{0}\}.
	\end{equation*}
	In \cite{Semal2021O}, one introduced the following definition:
	\begin{definition}\label{definition Hypothese H0}
		A group $G\leq\Aut(T)$ is said to satisfy the hypothesis \ref{Hypothese H0} if for all $\mathcal{T},\mathcal{T'}\in \mathfrak{T}_{0}$ we have that
		\begin{equation}\tag{$H_0$}\label{Hypothese H0}
		\Fix_G(\mathcal{T}')\leq \Fix_G(\mathcal{T})\mbox{ if and only if }\mathcal{T}\subseteq \mathcal{T}'. 
		\end{equation}
	\end{definition}
	\begin{lemma}[{\cite[Lemma 4.12]{Semal2021O}}]\label{Lemme la forme des elements de Sl}
			Let $G\leq \Aut(T)$ be a closed non-discrete unimodular subgroup satisfying the hypothesis \ref{Hypothese H0}. Then, $\mathcal{S}_{0}$ is a generic filtration of $G$ and the sets $\mathcal{S}_{0}\lb l \rb$ can be described as follows:
		\begin{itemize}
			\item If $l$ is even $\mathcal{S}_{0}\lb l\rb = \{\Fix_G(B_T(e,\frac{l}{2}))\lvert e\in E(T)\}.$
			\item If $l$ is odd $\mathcal{S}_{0}\lb l \rb=\{ \Fix_G(B_T(v,(\frac{l+1}{2})))\mid v\in V(T)\}.$
		\end{itemize} 
	\end{lemma}
	We come back to our case  $G=G^+_{(i)}(Y_0,Y_1)$. 
	\begin{lemma}\label{lemma G satisfies hypotheses H0}
		The group $G^+_{(i)}(Y_0,Y_1)$ satisfies the hypothesis \ref{Hypothese H0}.
	\end{lemma}
	\begin{proof}
		For every set $X$, let $\Delta_X$ denote the diagonal $\Delta_X=\{(x,x)\lvert x\in X\}\subseteq X\times X$. We choose two functions 
		\begin{equation*}
		\psi_0 \fct{\{1,...,d_0\}\times \{1,...,d_0\}-\Delta_{ \{1,...,d_0\}}}{\Sym(d_0)}{(k,l)}{\psi_0(k,l)}
		\end{equation*}
		\begin{equation*}
		\psi_1 \fct{\{1,...,d_1\}\times\{1,...,d_1\}-\Delta_{ \{1,...,d_0\}}}{\Sym(d_1)}{(k,l)}{\psi_1(k,l)}
		\end{equation*}
		such that $\psi_t(k,l)$ is a non-trivial element of $\Alt(d_t)$ which fixes $k$ but not $l$. Notice that the existence of such functions is guaranteed from the fact that  $d_0,d_1\geq 4$. For shortening of the formulation we denote by $G$ be the group $G^+_{(i)}(Y_0,Y_1)$. Let $\mathcal{T},\mathcal{T}'$ be two subtrees of $\mathfrak{T}_0$. If $\mathcal{T}\subseteq \mathcal{T}'$, we clearly have that $\Fix_G(\mathcal{T}')\leq \Fix_G(\mathcal{T})$. Now, let us suppose that $\mathcal{T}\not \subseteq \mathcal{T}'$. In order to prove that $G$ satisfies the hypothesis \ref{Hypothese H0}, we need to show that $\Fix_G(\mathcal{T}')\not \subseteq \Fix_G(\mathcal{T})$. Since $\mathcal{T}\not \subseteq \mathcal{T}'$, there exists a vertex $v\in \mathcal{T}$ that does not belong to $\mathcal{T}'$. Let $\gamma$ be the smallest geodesic from $v$ to $\mathcal{T}'$, let $v'$ be the vertex of $\gamma$ that is adjacent to $v$ and let $t\in\{0,1\}$ be such that $v'\in V_t$. Let $w$ be the neighbour of $v'$ which is the closest to $\mathcal{T}'$. Notice that this vertex exists and is unique since $\mathcal{T}$ and $\mathcal{T}'$ are complete. The definition of $v'$ ensures that $\mathcal{T}'\subseteq T(v',v)=\{x\in V(T)\lvert d_T(x,v')<d_T(x,v)\}$. Now, notice the existence of an automorphism $g\in \Fix_{\Aut(T)^+}(T(v',v))$ such that $\sigma_{(i)}(g,v')=\psi_t(i(w),i(v))$ and $\sigma_{(i)}(g,x)$ is even for every $x\in V(T)$. In particular, $g\in G^+_{(i)}(Y_0,Y_1)\cap \Fix_{\Aut(T)^+}(\mathcal{T}')$. However, $g$ does not fix $v$ by construction. This implies that $g\not\in \Fix_G(\mathcal{T})$.  		
	\end{proof}
	\noindent In particular, Lemma \ref{Lemme la forme des elements de Sl} ensures that $\mathcal{S}_0$ is a generic filtration of $G^+_{(i)}(Y_0,Y_1)$. 
	\subsection{Factorization}\label{Section factorization}
	
	Let $T$ be a $(d_0,d_1)$-semi-regular tree with $d_0,d_1\geq 4$ and let $V(T)=V_0\sqcup V_1$ be the associated bipartition. Let $i$ be a legal coloring of $T$ and let $Y_0,Y_1\subseteq \N$ be two finite subsets. We have shown in Section \ref{chapter generic filtration for Gi+YY}, that $\mathcal{S}_0$ is a generic filtration of  $G_{(i)}^+(Y_0,Y_1)$. The purpose of the present section is to prove that this generic filtration factorizes$^+$ at all depth $l\geq 1$.
	
	We start with some notations that will be used in the proof. For every two distinct vertices $v,w\in V(T)$, let $\lb v,w\rb$ be the unique geodesic between $v$ and $w$. Suppose that $d(v,w)=n$, let $v=v_0,v_1,...,v_n=w$ be the sequence of vertices corresponding to $\lb v,w\rb$ in $T$, let 
	\begin{equation*}\label{definition Proj sur la geor}
	p_{\lb v,w\rb}\fct{\lb v,w\rb -\{v\}}{\lb v,w\rb}{v_i}{v_{i-1}}
	\end{equation*}
	and let
	\begin{equation*}
	\begin{split}
	T(v,w)&= \{x\in V(T)\lvert d_T(x,p_{\lb v,w\rb }(w))<d_T(x,w)\}\\
	&= \{x\in V(T)\lvert d_T(x,v_{n-1})<d_T(x,w)\}.
	\end{split}
	\end{equation*}
	\begin{figure}[H]\label{drawing00}\caption{The set $T(v,w)$}
		\begin{center}
			\includegraphics[scale=0.09]{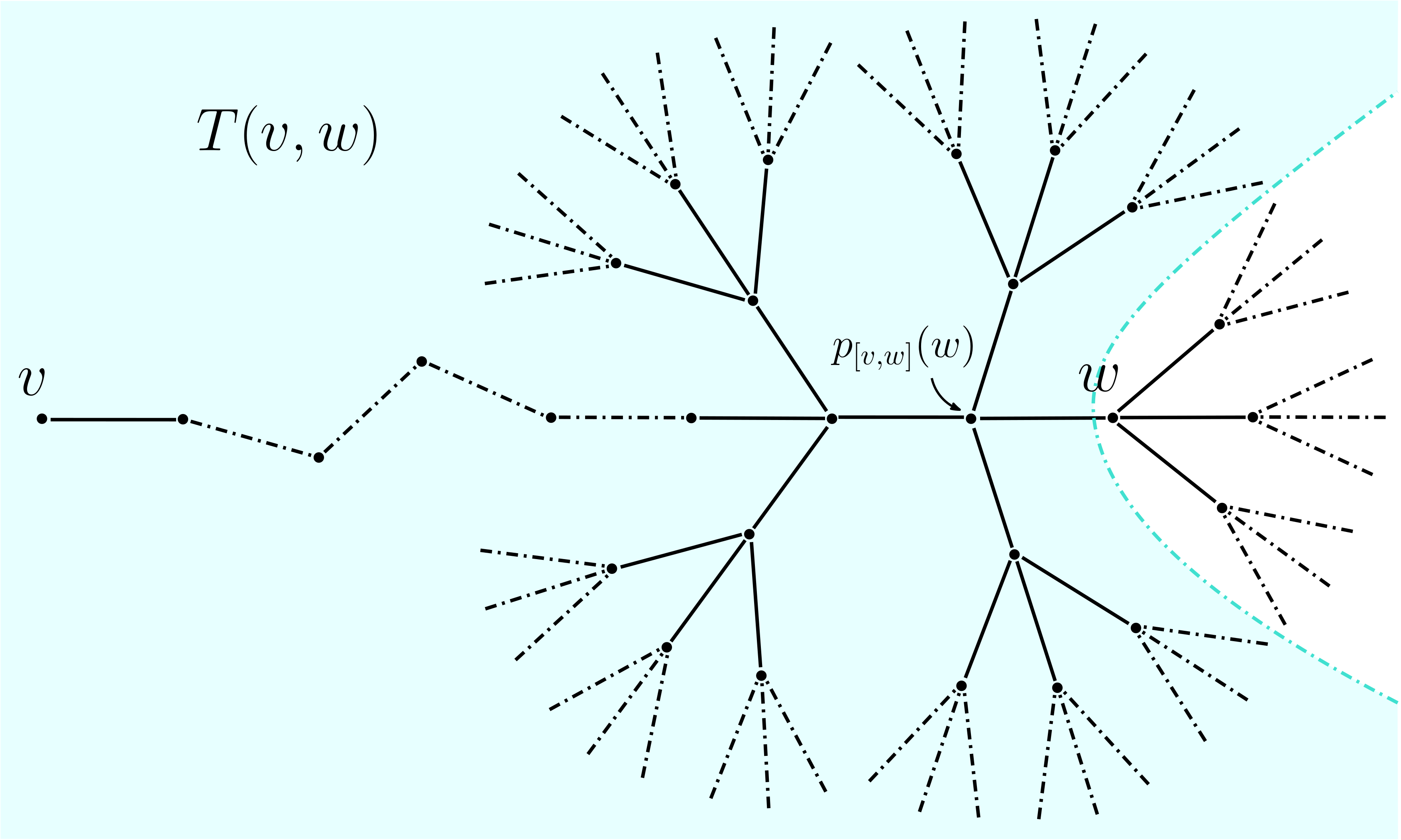}
		\end{center}
	\end{figure}
	\noindent The following intermediate result is the key ingredient required to prove the factorization of the generic filtration $\mathcal{S}_0$ of $G_{(i)}^+(Y_0,Y_1)$ at all depth $l\geq 1$.
	\begin{proposition}\label{Proposition la premiere etape de la fcatorization pour GY1Y_2}
		For all $l,l'\in \N$ such that $l\geq 1$ and $l'\geq l$, for all $U$ in the conjugacy class of an element of $\mathcal{S}_0\lb l \rb$ and every $V$ in the conjugacy class of an element of $\mathcal{S}_0\lb l'\rb$ such that $V\not \subseteq U$, there exists a subgroup $W\in\mathcal{S}_0\lb l-1\rb$ such that $U\subseteq W\subseteq VU$. 
	\end{proposition}
	\begin{proof}
		To shorten the proof and for clarity of the argument, parts of the reasoning are proved in Lemmas \ref{Lemme Omega 0 est non vide} and \ref{Lemme Omega n est non vide} below. Since the proof is quite long and technical, we start by giving an idea of its structure. We begin the proof by identifying the group $W$ from $U$ and $V$. We then prove that each element of $W$ decomposes as a product of an element of $V$ and an element of $U$. The proof of this decomposition is where the technicalities come from. It is achieved by a compactness argument taking advantage from the fact that $G_{(i)}^+(Y_0,Y_1)$ is defined by local actions conditions.
		
		Let $G= G_{(i)}^+(Y_0,Y_1)$. As announced at the beginning of the proof, we start by identifying $W$. Notice that $\mathfrak{T}_0$ is stable under the action of $G$. Furthermore, for every $g\in \Aut(T)^+$ and for every subtree $\mathcal{T}$ of $T$, we have that $g\Fix_G(\mathcal{T})g^{-1}=\Fix_G(g\mathcal{T})$. In particular, there exist $\mathcal{T}$, $\mathcal{T}'\in \mathfrak{T}_0$ such that $U=\Fix_G(\mathcal{T})$ and $V=\Fix_G(\mathcal{T}')$. Since $V\not \subseteq U$, notice that $\mathcal{T}\not \subseteq \mathcal{T'}$. If $l$ is even, Lemma \ref{Lemme la forme des elements de Sl} ensures that $\mathcal{T}=B_T(e,\frac{l}{2})$ for some edge $e\in E(T)$. Furthermore, since $\mathcal{T}\not \subseteq \mathcal{T'}$ and since $l'\geq l$, there exists a unique vertex $v\in e$ such that $\mathcal{T}'\subseteq T(v,w)\cup B_T(v,\frac{l}{2})$ where $w$ denotes the other vertex of $e$. In that case, we let $\mathcal{T}_W=B_T(v,\frac{l}{2})$. If on the other hand $l$ is odd, Lemma \ref{Lemme la forme des elements de Sl} ensures that $\mathcal{T}=B_T(w,\frac{l+1}{2})$ for some vertex $w\in V(T)$. Furthermore, since $\mathcal{T}\not \subseteq \mathcal{T'}$ and since $l'\geq l$, there exists a unique vertex $v\in B_T(w,1)-\{w\}$ such that $\mathcal{T}'\subseteq T(v,w)\cup B_T(\{v,w\},\frac{l-1}{2})$. In that case, we let $\mathcal{T}_W=B_T(\{v,w\},\frac{l-1}{2})$. In both cases, we set $W=\Fix_G( \mathcal{T}_W)$. By construction, notice that $W\in \mathcal{S}_0\lb l-1\rb$ and that $U\subseteq W$ (since $\mathcal{T}_W\subseteq \mathcal{T}$). Our purpose is therefore to show that $W\subseteq V U$. To this end, let $\alpha\in W$ and let us show the existence of an element $\alpha_0\in U$ such that $\restr{\alpha}{\mathcal{T}'}=\restr{\alpha_0}{\mathcal{T}'}$. We start by explaining why the existence of $\alpha_0$ settles the proof. Indeed, if $\alpha_0$ exists, notice that the automorphism $\alpha_1=\alpha_0^{-1}\circ \alpha$ is an element of $G$ for which $\restr{\alpha_1}{\mathcal{T}'}=\restr{id}{\mathcal{T}'}$. In particular, we have that $\alpha_1\in \Fix_G(\mathcal{T}')$, $\alpha_0\in \Fix_G(\mathcal{T})$ and by construction $\alpha=\alpha_0 \circ \alpha_1$ which proves that $W\subseteq UV$. Applying the inverse map on both sides of the inclusion we obtain that $W\subseteq VU$ which settles the proof.  
		
		Now, let us prove the existence of $\alpha_0$. As announced, at the beginning of the proof, we are going to use a compactness argument taking advantage from the fact that $G_{(i)}^+(Y_0,Y_1)$ is defined by local actions conditions. To be more precise, we are going to define a descending chain of non-empty compact sets $\Omega_n\subseteq \Aut(T)^+$ and an increasing chain of finite subtrees $\mathcal{R}_n$ of $T$ such that $T=\bigcup_{n\in \N} \mathcal{R}_n$ and such that for all $h\in \Omega_n$ we have:
		\begin{itemize}
			\item $h\in \Fix_G(\mathcal{T})$ and $\restr{h}{\mathcal{T}'}=\restr{\alpha}{\mathcal{T}'}$
			\item $\Sgn_{(i)}(h,S_{Y_0}(v))=1$ for all $v$ in $V_{t_0}\cap \mathcal{R}_n$.
			\item $\Sgn_{(i)}(h,S_{Y_1}(v))=1$ for all  $v$ in $V_{t_1}\cap \mathcal{R}_n$.
		\end{itemize} 
		We recall that in the above $t_0=\max(Y_0)\mod2$,  $t_1=(1+\max(Y_1))\mod2$ and $\max(\es)=0$. Let us first show that this settles the existence of $\alpha_0$. Since the $\Omega_n$ form a descending chain of non-empty compact sets in a Hausdorff space we obtain $\bigcap_{n\in \N} \Omega_n\not =\es$. Let $\alpha_0 \in \bigcap_{n\in \N} \Omega_n$. Since $\alpha_0\in \Omega_0$, notice that $ \restr{\alpha_0}{\mathcal{T}}=\restr{\id}{\mathcal{T}}$, $\restr{\alpha_0}{\mathcal{T}'}=\restr{\alpha}{\mathcal{T}'}$. To see that $\alpha_0$ is as desired, we are left to show that $\alpha_0\in G_{(i)}^+(Y_0,Y_1)$. However, for every $v\in V_{t}(T)$, there exists a positive integer $n\in \N$ such that $v\in \mathcal{R}_n$ and since $\alpha_0\in \Omega_{n}$ we have that $\Sgn_{(i)}(\alpha_0,S_{Y_t}(v))=1$. This proves that $\alpha_0$ is as desired.
		
		We are left to define the descending chain of non-empty compact sets $\Omega_n\subseteq \Aut(T)^+$. Suppose that $\max(Y_0)\leq \max(Y_1)$ (the proof for $\max(Y_1)\leq \max(Y_0)$ is similar). Let $\gamma$ be the smallest geodesic of $T$ containing both the centre of $\mathcal{T}$, the centre of $\mathcal{T}'$ and oriented from $\mathcal{T}$ to $\mathcal{T}'$ (note that the centre is either a vertex or an edge depending on the values of $l$ and $l'$). Since $\mathcal{T}\not \subseteq \mathcal{T}'$ and since $l'\geq l$ notice that $\gamma $ contains at least two vertices. 
		\begin{figure}[H]\label{drawing01}\caption{The tree $\mathcal{T}_W$ and the geodesic $\gamma$}
				\begin{center}
					\includegraphics[scale=0.08]{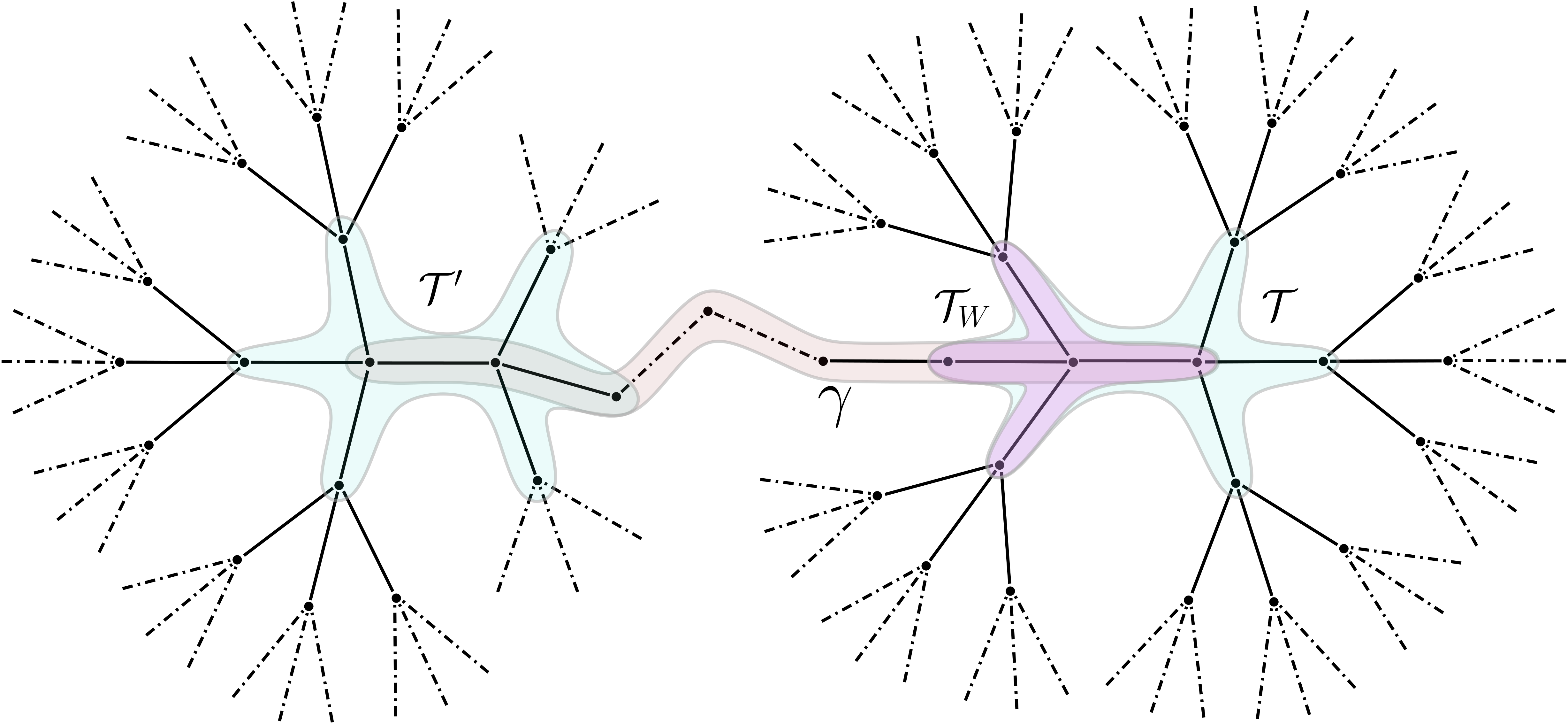}
				\end{center}
		\end{figure}
		The increasing chain of finite subtrees $\mathcal{R}_n$ of $T$ such that $T=\bigcup_{n\in \N} \mathcal{R}_n$ that we are going to use is $\mathcal{R}_n=B_T(\gamma,n)$. We let
		\begin{equation*}
		\Omega_{-1}\qq=\qq \left\{ h \in \Aut(T)^+ \middle\vert \begin{array}{l}\restr{h}{\mathcal{T}}=\restr{\id}{\mathcal{T}} \mbox{ and } \restr{h}{\mathcal{T}'}=\restr{\alpha}{\mathcal{T}'} 
		\end{array}\right\}.
		\end{equation*}
		Since $\alpha\in \Fix_G(\mathcal{T}_W)$ and since $\mathcal{T}_W$ contains every vertices of $\mathcal{T}\cap \mathcal{T}'$ notice that $\Omega_{-1}$ is not empty. Now, since $\max(Y_0)\leq \max(Y_1)$, notice that there exists a unique $r\in \N$ such that $\max(Y_0)+2r\leq \max(Y_1)\leq \max(Y_0)+2r +1$ (where one of this inequality is an equality). We let
		\begin{equation*}
		\Omega_0=\left\{h\in \Omega_{-1}\ \middle\vert \begin{array}{l}
		\Sgn_{(i)}(g,S_{Y_0}(v))=1\qq \mbox{ for each }v\in B_T(\gamma,2r)\cap V_{t_0},\\ 
		\Sgn_{(i)}(g,S_{Y_1}(v))=1\qq \mbox{ for each }v\in B_T(\gamma,0)\cap V_{t_1}
		\end{array}\right\}.
		\end{equation*} 
		Lemma \ref{Lemme Omega 0 est non vide} below ensures that this set is not empty. From there, we define the sets $\Omega_n$ by induction on $n$. For every $n\geq 1$, let $h_n$ an element of $\Omega_{n-1}$ and let
		\begin{equation*}
		\Omega_n\qq=\qq \left\{ h \in \Omega_{n-1}\qq \ \middle\vert \begin{array}{l}
		\q\q \restr{h}{B_T(\gamma,n+{\max(Y_1)})}=\restr{h_{n}}{B_T(\gamma,n+{\max(Y_1)})},\\
		\qq\Sgn_{(i)}(h,S_{Y_0}(w))=1\q \forall w\in B_T(\gamma,n+2r)\cap V_{t_0},\\
		\qq\Sgn_{(i)}(h,S_{Y_1}(w))=1\q \forall w\in B_T(\gamma,n)\cap V_{t_1}
		\end{array}\right\}.
		\end{equation*} 
		For this induction to make sense, it is important for $\Omega_{n}$ to be not empty for all $n\geq 1$. This is proved by Lemma \ref{Lemme Omega n est non vide} below which ensures that $\Omega_{n}$ is a non-empty compact set. The result follows.
	\end{proof}

	Our current purpose is to prove Lemmas \ref{Lemme Omega 0 est non vide} and \ref{Lemme Omega n est non vide}. To this end, we introduce some formalism that will be used in both proofs. For all $v\in V(T)$, we are going to need an automorphism $h_{(v)}\in \Aut(T)^+$ that will be used to create an element of $\omega_{n+1}$ from an element of $\Omega_n$. We start by choosing four functions:
	\begin{equation*}
	\phi_0 \fct{\{1,...,d_0\}}{\Sym(d_0)}{k}{\phi_0(k)}
	\end{equation*}
	\begin{equation*}
	\phi_1 \fct{\{1,...,d_1\}}{\Sym(d_1)}{k}{\phi_1(k)}
	\end{equation*}
	\begin{equation*}
	\tilde{\phi}_0 \fct{\{1,...,d_0\}\times\{1,...,d_0\}}{\Sym(d_0)}{(k,l)}{\tilde{\phi}_0(k,l)}
	\end{equation*}
	\begin{equation*}
	\tilde{\phi}_1 \fct{\{1,...,d_1\}\times\{1,...,d_1\}}{\Sym(d_1)}{(k,l)}{\tilde{\phi}_1(k,l)}
	\end{equation*}
	such that $\phi_t(k)$ is an odd permutation of $\Sym(d_t)$ which fixes $k$ and $\tilde{\phi}_t(k,l)$ is an odd permutation of $\Sym(d_t)$ which fixes $k$ and $l$. 
	
	If $v\in V(T)-\gamma$ we choose $w\in \gamma$ and let $h_{(v)}\in \Aut(T)^+$ be such that:
		\begin{enumerate}
			\item $h_{(v)}\in \Fix_{\Aut(T)^+}(T(p_{\lb w,v\rb }(v),v))$.
			\item $\sigma_{(i)}(h_{(v)},v)=\phi_t(i(p_{[w,v]}(v)))$ where $t\in\{0,1\}$ is such that $v\in V_t$.
		\end{enumerate}
	Notice that for all $v\in V(T)-\gamma$ and every $w,w'\in \gamma$ $p_{\lb w,v\rb}(v) = p_{\lb w',v\rb}(v)$ (recall the definition of $p_{\lb w,v\rb}$ from page \pageref{definition Proj sur la geor}) so that our choice of $w\in \gamma$ does not change the two properties that $h_{(v)}$ must satisfy. 
		\begin{figure}[H]\label{drawing02}\caption{The automorphism $h_{(v)}$}
			\begin{center}
				\includegraphics[scale=0.08]{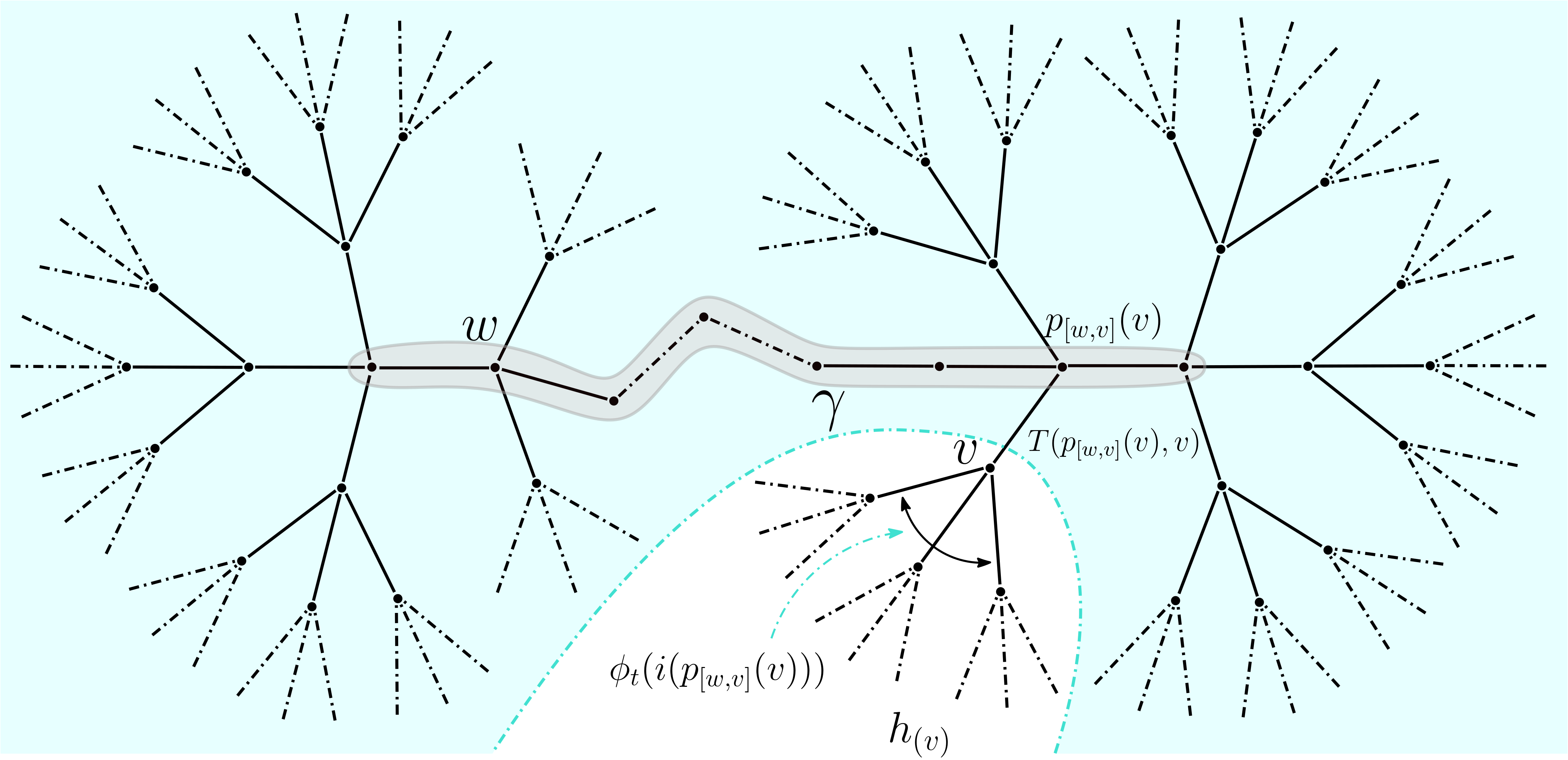}
			\end{center}
		\end{figure}
	
	If $v\in \gamma$, we have two cases. Remember that $\gamma$ has at least two vertices. If $v$ is an end of $\gamma$ let $w$ be the unique vertex of $\gamma$ that is adjacent to $v$ and choose an automorphism $h_{(v)}\in \Aut(T)^+$ such that:
	\begin{enumerate}
		\item $h_{(v)}\in \Fix_{\Aut(T)^+}(T(w,v))$.
		\item $\sigma_{(i)}(h_{(v)},v)=\phi_t(i(w))$ where $t\in\{0,1\}$ is such that $v\in V_t$.
	\end{enumerate}
	On the other hand, if $v$ is not an end of $\gamma$, let $w_1,w_2$ be the two neighbours of $v$ which belong to $\gamma$ and choose an automorphism $h_{(v)}\in \Aut(T)^+$ such that:
	\begin{enumerate}
		\item $h_{(v)}\in\Fix_{\Aut(T)^+}(T(w_1,v)\cup T(w_2,v))$.
		\item $\sigma_{(i)}(h_{(v)},v)=\tilde{\phi}_t(i(w_1),i(w_2))$ where $t\in\{0,1\}$ is such that $v\in V_t$.
	\end{enumerate}
	We are now ready to prove  Lemmas \ref{Lemme Omega 0 est non vide} and \ref{Lemme Omega n est non vide}.
	\begin{lemma}\label{Lemme Omega 0 est non vide}
		The set
		\begin{equation*}
		\Omega_0=\left\{h\in \Omega_{-1}\ \middle\vert \begin{array}{l}
		\Sgn_{(i)}(g,S_{Y_0}(v))=1\qq \mbox{ for each }v\in B_T(\gamma,2r)\cap V_{t_0},\\ 
		\Sgn_{(i)}(g,S_{Y_1}(v))=1\qq \mbox{ for each }v\in B_T(\gamma,0)\cap V_{t_1}
		\end{array}\right\}.
		\end{equation*} 
		is not empty.
	\end{lemma}
	\begin{proof}
		We recall that $\max(Y_0)\leq \max(Y_1)$, that $r\in \N$ is the unique integer such that $\max(Y_0)+2r\leq \max(Y_1)\leq \max(Y_0)+2r +1$, that $t_0=\max(Y_0) \mod 2$ and that $t_1=\max(Y_1)+1 \mod 2$. Remember from the proof of Proposition \ref{Proposition la premiere etape de la fcatorization pour GY1Y_2} that $\Omega_{-1}$ is not empty and let $h_0\in \Omega_{-1}$. We are going to modify the element $h_0$ with the automorphisms $h_{(v)}$ in order to obtain an element of $\Omega_{0}$. A concrete example of the procedure is given on a $4$-regular tree with $Y_0=\{0\}$ and $Y_1=\{1,2\}$ by figures \ref{drawing03}, \ref{drawing04} and \ref{drawing05}. In those figures:
		\begin{itemize}
			\item The hollow vertices are those concerned by the current and previous steps.
			\item The vertices circled in purple are the vertices for which we desire to change the sign $\Sgn_{(i)}(g,S_{Y_0}(v))$ or $\Sgn_{(i)}(g,S_{Y_1}(v))$ (depending on the step) without affecting the sign of other hollow vertices.
			\item The vertices circled in yellow are the vertices for which a change of the local action is applied in order to fulfil the desired change of sign (note that for our choice $Y_0=\{0\}$ those vertices are also the vertices circled in purple). 
		\end{itemize}
		Let $\{w_{0,0},...,w_{0,m_0}\}$ be the set of vertices $w\in B_T(\gamma,0)\cap V_{t_0}$ such that $\Sgn_{(i)}(h_0,S_{Y_0}(w))=-1$. For all $j=0,1,...,m_0$ we choose a vertex $$v_{0,j}\in \bigcap_{w\in \gamma-\{w_{0,j}\}}{}T(w_{0,j},w)$$ such that $d(v_{0,j},w_{0,j})=\max(Y_0)$. 	
		In particular, notice that $v_{0,j}\in S_{Y_0}(w_{0,j})$ but that $v_{0,j}\not \in S_{Y_0}(w)$ for every $w \in B_T(\gamma,0)\cap V_{t_0} $. Furthermore, since $\Sgn_{(i)}(h_0,S_{Y_0}(w_{0,j}))=-1$, $\restr{h_0}{\mathcal{T}}= \restr{\id}{\mathcal{T}}$, $\restr{h_0}{\mathcal{T}'}= \restr{\alpha}{\mathcal{T}'}$ and due to the form of $\mathcal{T}$ and $\mathcal{T}'$, the vertices $v_{0,j}$ must be such that the automorphisms $h_{(v_{0,0})},..., h_{(v_{0,m_0})}$ fix $\mathcal{T}\cup \mathcal{T}'$ pointwise. In particular, the automorphism 
		$$h_{0,0}= h_0 \circ h_{(v_{0,0})}\circ ...\circ h_{(v_{0,m_0})}$$ satisfies  $\restr{h_{0,0}}{\mathcal{T}}=\restr{h_0}{\mathcal{T}}=\restr{\id}{\mathcal{T}}$, $ \restr{h_{0,0}}{\mathcal{T}'}=\restr{h_0}{\mathcal{T}'}=\restr{\alpha}{\mathcal{T}'}$ and $$\Sgn_{(i)}(h_{0,0},S_{Y_0}(w))=1\qq \forall w\in \gamma\cap V_{t_0}.$$
		\begin{figure}[H]\caption{Step I of the proof of Lemma \ref{Lemme Omega 0 est non vide}} \label{drawing03}
			\includegraphics[scale=0.08]{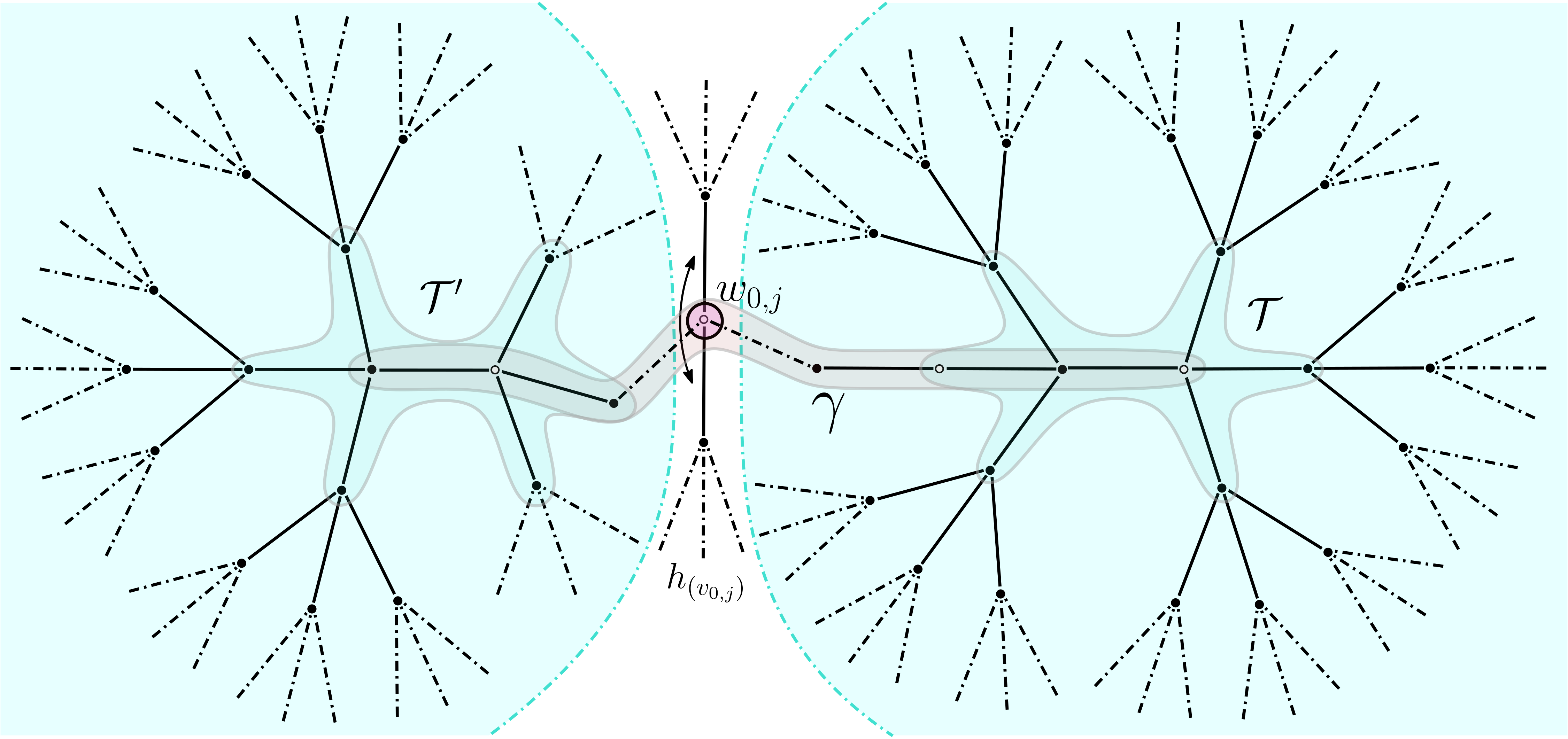}
		\end{figure}
		
		If $r\not=0$, we iterate this procedure. For every $1\leq \nu\leq 2r$, let $\{w_{\nu,0},...,w_{\nu,m_{\nu}}\}$ be the set of vertices $w\in B_T(\gamma,\nu)\cap V_{t_0}$ such that $$\Sgn_{(i)}(h_{\nu-1,0},S_{Y_0}(w))=-1.$$ 
		For all $j=0,1,...,m_\nu$ we choose a vertex $$v_{\nu,j}\in \bigcap_{w\in B_T(\gamma,\nu)\cap V_{t_0}-\{w_{\nu,j}\}}{}T(w_{\nu,j},w)$$ such that $d(v_{\nu,j},w_{\nu,j})=\max(Y_0)$. Hence, notice that $v_{\nu,j}\in S_{Y_0}(w_{\nu,j})$ but that $v_{\nu,j}\not \in S_{Y_0}(w)$ for every $w\in B_T(\gamma,\nu)\cap V_{t_0}-\{w_{\nu,j}\}$. Furthermore, notice that the automorphisms $h_{(v_{\nu,0})}, ..., h_{(v_{\nu,m_\nu})}$ fix $\mathcal{T}\cup \mathcal{T}'$ pointwise. In particular, the automorphism  
		$$h_{\nu,0}=h_{\nu-1,0} \circ h_{(v_{\nu,0})}\circ ...\circ h_{(v_{\nu,m_\nu})}$$ satisfies that  $\restr{h_{\nu,0}}{\mathcal{T}}=\restr{h_{\nu-1,0}}{\mathcal{T}}=\restr{\id}{\mathcal{T}}$, $ \restr{h_{\nu,0}}{\mathcal{T}'}=\restr{h_{\nu-1,0}}{\mathcal{T}'}=\restr{\alpha}{\mathcal{T}'}$ and $$\Sgn_{(i)}(h_{\nu,0},S_{Y_0}(w))=1\qq \forall w\in B_T(\gamma,\nu)\cap V_{t_0}.$$
		\begin{figure}[H]\caption{Step II of the proof of Lemma \ref{Lemme Omega 0 est non vide}} \label{drawing04}
			\includegraphics[scale=0.08]{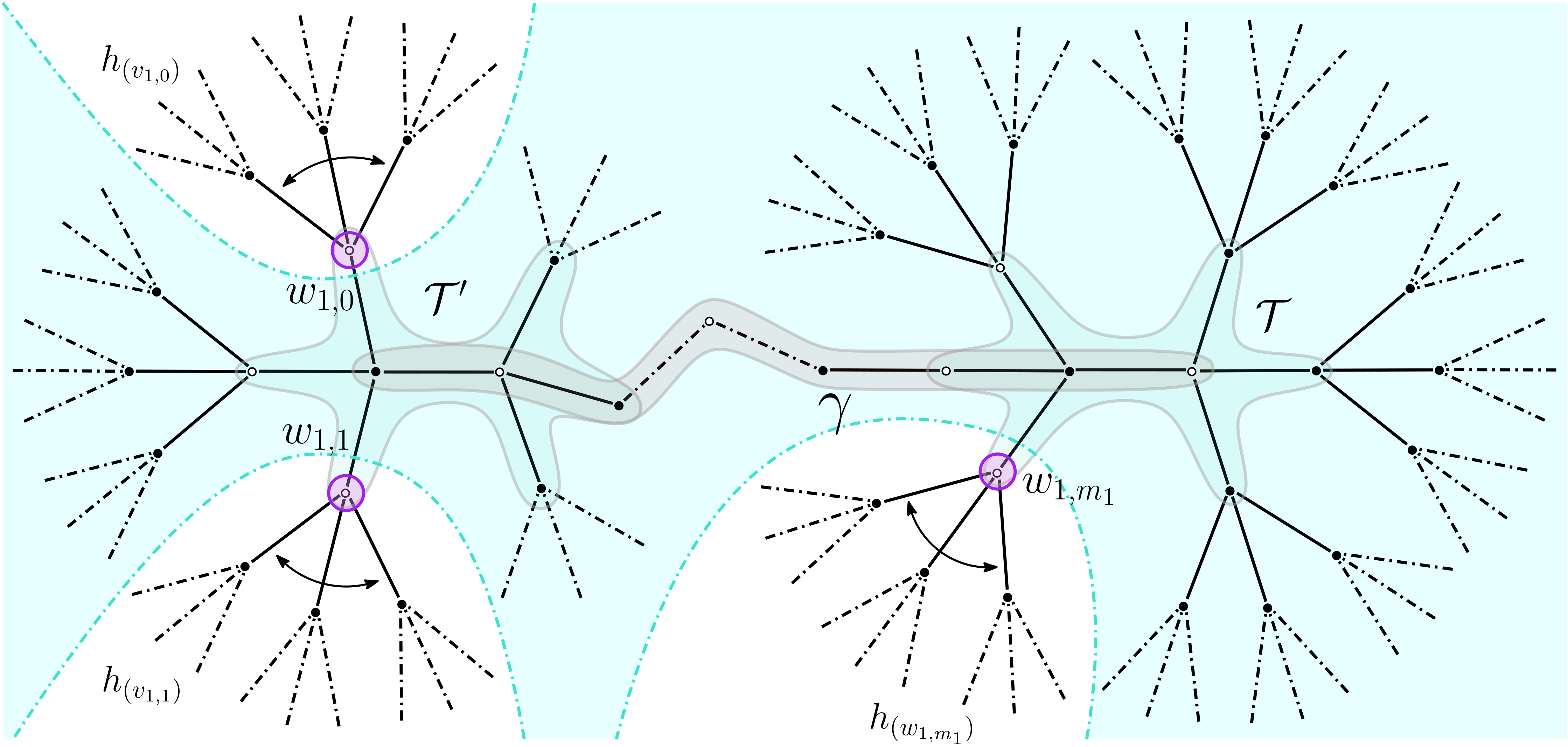}
		\end{figure}
		Consider the element $h_{2r,0}$ that we have just constructed. This element behaves as desired for the condition given by $Y_0$ on the vertices of $B_T(\gamma,2r)\cap V_{t_0}$. However, nothing ensures that the condition given by $Y_1$ on the vertices of $B_T(\gamma,0)\cap V_{t_1}$ is yet satisfied. We now take care of this task. Let $\{w_0,...,w_m\}$ be the set of vertices $w\in B_T( \gamma,0)\cap V_{t_1}$ such that $\Sgn_{(i)}(h_{2r,0},S_{Y_1}(w))=-1$. For all $j=0,1,...,m$ let $v_j\in\bigcap_{w\in \gamma-\{w_j\}} T(w_j,w)$ such that $d(v_j,w_j)=\max(Y_1)$.
		In particular, notice that $v_{j}\in S_{Y_1}(w_{j})$ but that $v_{j}\not \in S_{Y_1}(w)$ for every $w\in B_T(\gamma,0) \cap V_{t_1}$. Furthermore, since $\max(Y_0)+2r\leq \max(Y_1)$ notice  from Remark \ref{remark les sommets ont des types opposers} that $v_j\not \in S_{Y_0}(v)$ for every $v\in B_T(\gamma,2r)\cap  V_{t_0}$. On the other hand, just as before, the automorphisms $h_{(v_{0})}, ..., h_{(v_{m})}$ fix $\mathcal{T}\cup \mathcal{T}'$ pointwise. In particular, $h_{1}= h_{2r,0}\circ h_{(v_{0})}\circ ...\circ h_{(v_{m})}$ satisfies:
		\begin{itemize}
			\item $\restr{h_{1}}{\mathcal{T}}=\restr{h_{2r,0}}{\mathcal{T}}=\restr{\id}{\mathcal{T}}$ and $\restr{h_{1}}{\mathcal{T}'}=\restr{h_{2r,0}}{\mathcal{T}'}=\restr{\alpha}{\mathcal{T}'}$.
			\item $\Sgn_{(i)}(h_{1},S_{Y_0}(w))=1\qq \forall w\in B_T(\gamma,2r)\cap V_{t_0}.$
			\item $\Sgn_{(i)}(h_{1},S_{Y_1}(w))=1\qq \forall w\in B_T(\gamma,0)\cap V_{t_1}.$
		\end{itemize} 
		\begin{figure}[H]\caption{Step III of the proof of Lemma \ref{Lemme Omega 0 est non vide}} \label{drawing05}
			\includegraphics[scale=0.08]{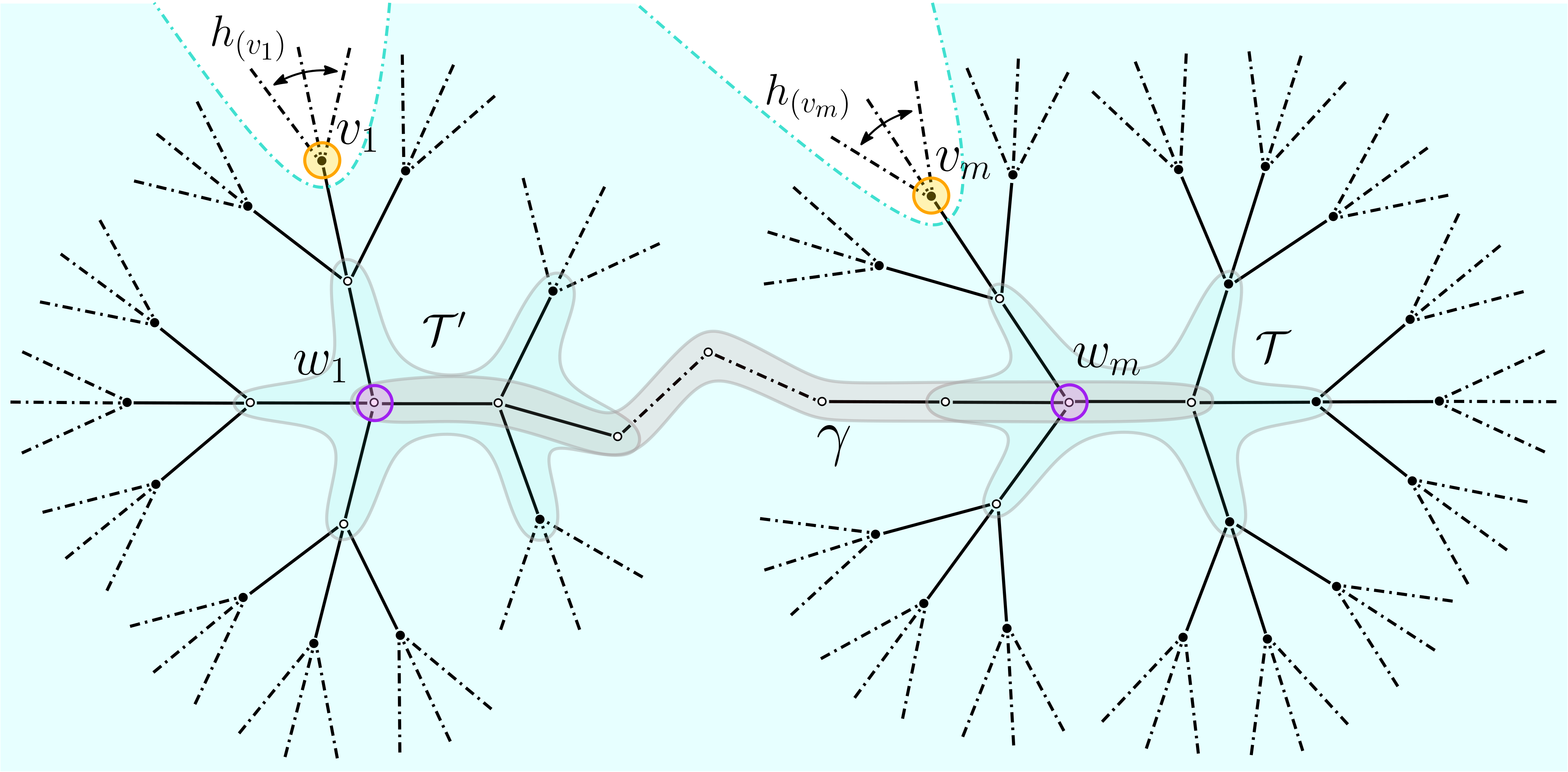}
		\end{figure}
		This proves that $h_1\in \Omega_0$ and therefore that $\Omega_0$ is not empty. 
	\end{proof}
	\begin{lemma}\label{Lemme Omega n est non vide}
		For all $n\geq 1$, the set	
		\begin{equation*}
		\Omega_n\qq=\qq \left\{ h \in \Omega_{n-1}\qq \ \middle\vert \begin{array}{l}
		\q\q \restr{h}{B_T(\gamma,n+{\max(Y_1)})}=\restr{h_{n}}{B_T(\gamma,n+{\max(Y_1)})},\\
		\qq\Sgn_{(i)}(h,S_{Y_0}(w))=1\q \forall w\in B_T(\gamma,n+2r)\cap V_{t_0},\\
		\qq\Sgn_{(i)}(h,S_{Y_1}(w))=1\q \forall w\in B_T(\gamma,n)\cap V_{t_1}
		\end{array}\right\}.
		\end{equation*}
		is a non-empty compact subsets of $\Aut(T)^+$. 
	\end{lemma}
	\begin{proof}
		We show that $\Omega_n$ is not empty by induction. Lemma \ref{Lemme Omega 0 est non vide} ensures that $\Omega_0$ is not empty. Suppose that $\Omega_{n-1}$ is not empty and let $h_n\in \Omega_{n-1}$ be the automorphism appearing in the definition of $\Omega_n$. Just as in the proof of Lemma \ref{Lemme Omega 0 est non vide}, we are going to modify $h_n$ with the automorphisms $h_{(v)}$ in order to obtain an element of $\Omega_{n}$. A concrete example of the procedure is given by figures \ref{drawing201} and \ref{drawing202} on a $4$-regular tree with $Y_0=\{0\}$ and even $\max{(Y_1)}$ (with the same conventions as before and where the vertices concerned by the current step are circled in pink). Let $\{\tilde{w}_0,...,\tilde{w}_{k}\}$ be the set of vertices $w\in B_T(\gamma,n+2r)\cap V_{t_0}$ such that $\Sgn_{(i)}(h_n,S_{Y_0}(w))=-1$. For each $j=0,1,...,k$ we choose a vertex $\tilde{v}_{j}\in \bigcap_{w\in B_T(\gamma,n+2r)\cap V_{t_0}-\{\tilde{w}_j\}}{}T(\tilde{w}_{j},w)$ such that $d(\tilde{v}_{j},\tilde{w}_{j})=\max(Y_0)$. Hence, notice that $\tilde{v}_j\not \in S_{Y_0}(w)$ for all $w\in B_T(\gamma,n+2r)\cap V_{t_0}-\{\tilde{w}_j\}$. Furthermore, since $\max(Y_1)-1\leq 2r+\max(Y_0)$ notice from Remark \ref{remark les sommets ont des types opposers} that $\tilde{v}_j\not \in S_{Y_1}(w)$ for all $w\in  B_T(\gamma,n-1)\cap V_{t_1}$. Furthermore, since $\Sgn_{(i)}(h_{n},S_{Y_0}(\tilde{w}_{j}))=-1$, $\restr{h_{n}}{\mathcal{T}}= \restr{\id}{\mathcal{T}}$, $\restr{h_{n}}{\mathcal{T}'}= \restr{\alpha}{\mathcal{T}'}$ and due to the form of $\mathcal{T}$ and $\mathcal{T}'$, notice that the automorphisms $h_{(\tilde{v}_{0})}, ..., h_{(\tilde{v}_{k})}$ fix $\mathcal{T}\cup \mathcal{T}'$ pointwise. In particular, $\tilde{h}_n=h_{n} \circ h_{(\tilde{v}_{0})}\circ ...\circ h_{(\tilde{v}_{k})}$ satisfies:
		\begin{itemize}
			\item $\restr{\tilde{h}_n}{B_T(\gamma,n-1+\max{(Y_1)})}=\restr{h_{n}}{B_T(\gamma,n-1+\max(Y_1))}.$
			\item $\Sgn_{(i)}(h_{n},S_{Y_0}(w))=1\qq \forall w\in B_T(\gamma,n+2r)\cap V_{t_0}.$
			\item $\Sgn_{(i)}(h_{n},S_{Y_1}(w))=1\qq \forall w\in B_T(\gamma,n-1)\cap V_{t_1}.$
		\end{itemize} 
		\begin{figure}[H]\caption{Step I of the proof of Lemma \ref{Lemme Omega n est non vide}}\label{drawing201}
			\includegraphics[scale=0.075]{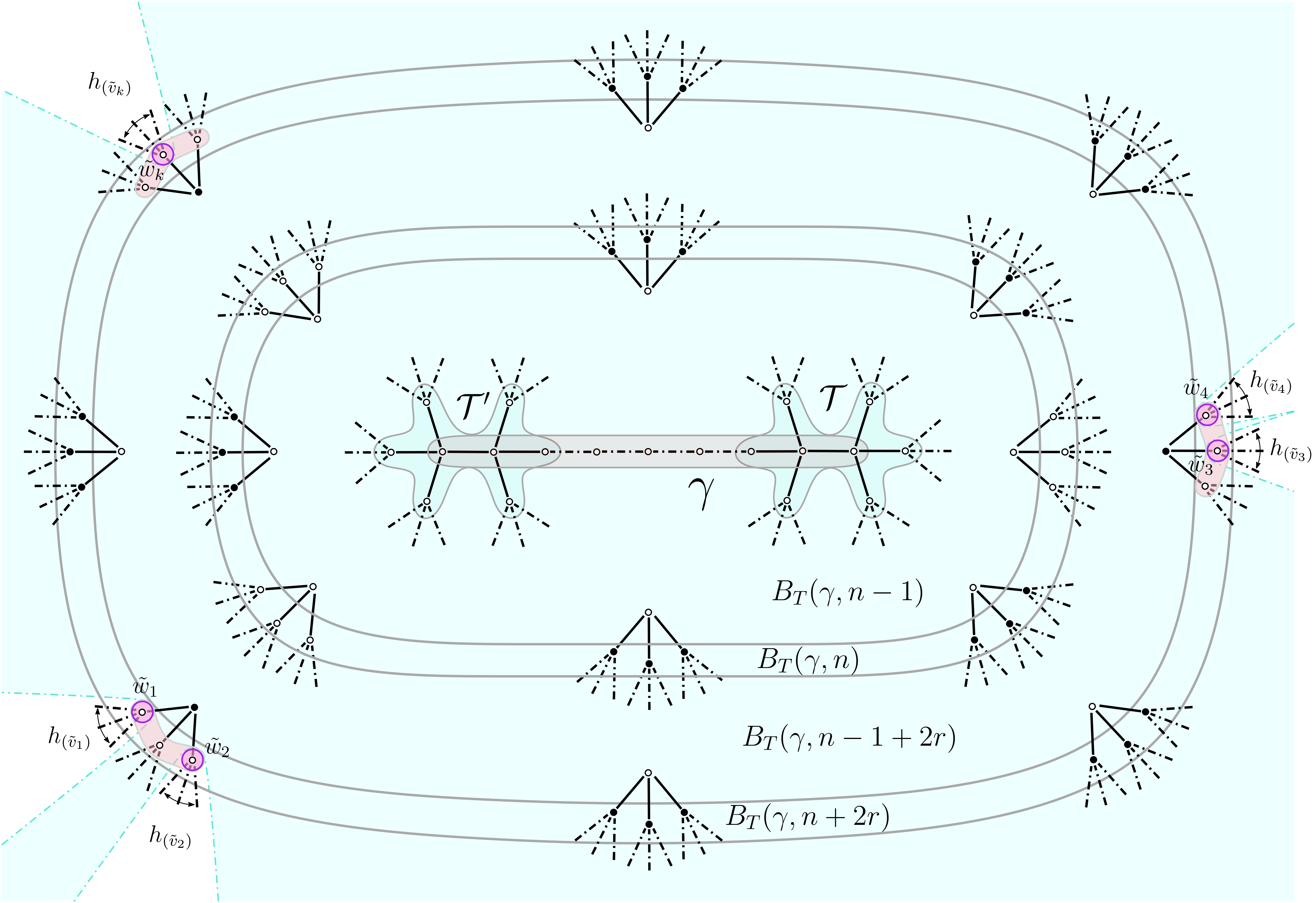}
		\end{figure}
		
		Now, let  $\{w_0,...,w_m\}$ be the set of vertices $w\in B_T(\gamma,n)\cap V_{t_1}$ such that $\Sgn_{(i)}(\tilde{h}_n,S_{Y_1}(w))=-1$. For each $j=0,1,...,m$, we choose a vertex $v_{j}\in \bigcap_{w\in B_T(\gamma,n)-\{w_j\}}{}T(w_{j},w)$ such that $d(v_{j},w_{j})=\max(Y_1)$. Since $\max(Y_0)+2r\leq\max(Y_1)$ notice from Remark \ref{remark les sommets ont des types opposers} that $v_j\not \in S_{Y_0}(w)$ for every $w\in B_T(\gamma,n+2r)\cap V_{t_0}$ and $v_j\not \in S_{Y_1}(w)$ for every $w\in B_T(\gamma,n)\cap V_{t_1}-\{w_j\}$. Just as before, notice that the automorphisms $h_{(v_{0})}, ...,h_{(v_{m})}$ fix $\mathcal{T}\cup \mathcal{T}'$ pointwise. In particular, $h_{n+1}=\tilde{h}_{n} \circ h_{(v_{0})}\circ ...\circ h_{(v_{m})}$ satisfies:
		\begin{itemize}
			\item $\restr{h_{n+1}}{B_T(\gamma,n-1+{\max(Y_1)})}=\restr{h_{n}}{B_T(\gamma,n-1+{\max(Y_1)})}.$
			\item $\Sgn_{(i)}(h_{n+1},S_{Y_0}(w))=1\qq \forall w\in B(\gamma,n+2r)\cap V_{t_0}.$
			\item $\Sgn_{(i)}(h_{n+1},S_{Y_1}(w))=1\qq \forall w\in B(\gamma,n)\cap V_{t_1}.$
		\end{itemize}
		\begin{figure}[H]\caption{Step II of the proof of Lemma \ref{Lemme Omega n est non vide}}\label{drawing202}
			\includegraphics[scale=0.075]{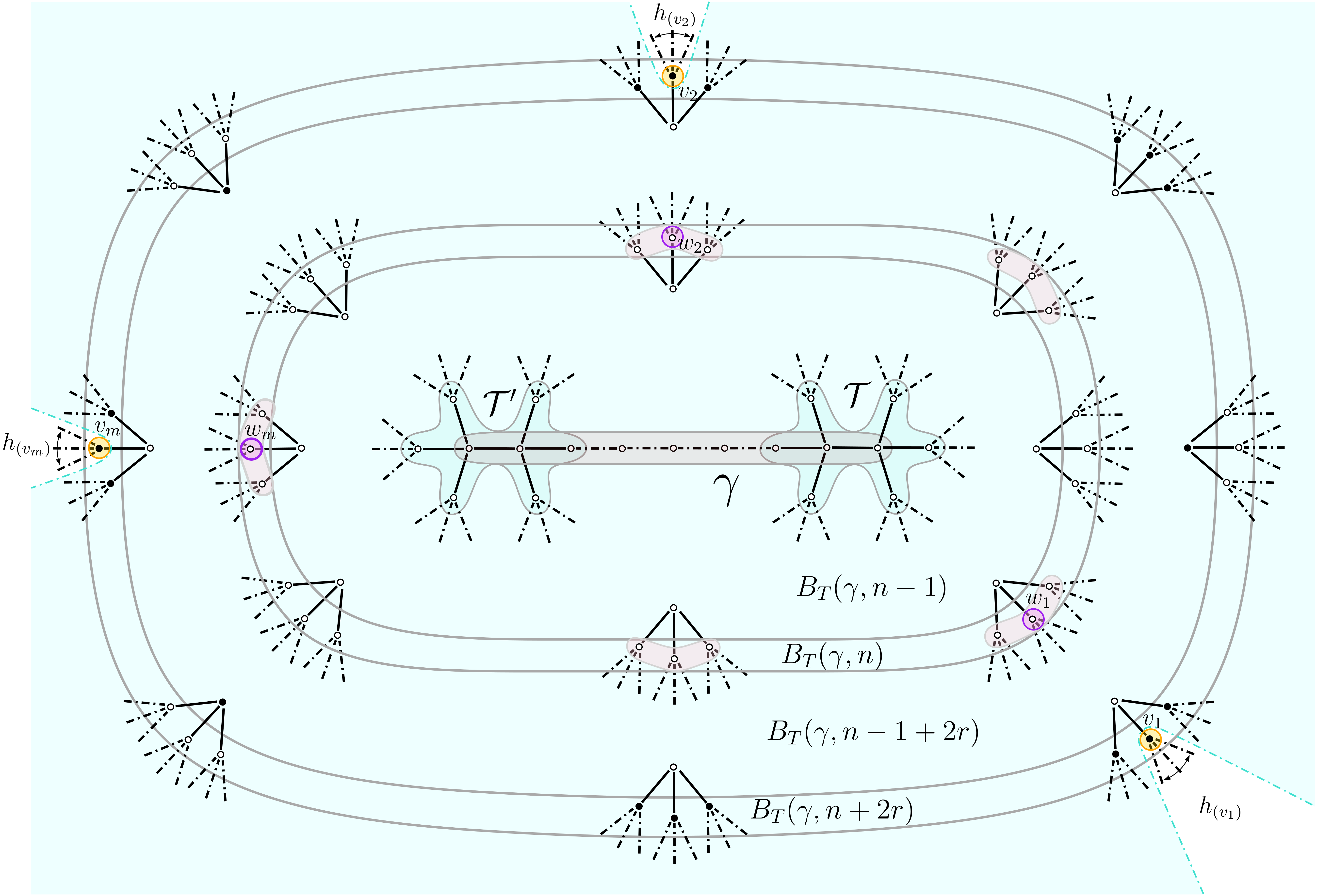}
		\end{figure}
		This proves that $\Omega_n$ is not empty. We now show that $\Omega_n$ is compact for every integer $n\geq 1$. To this end, notice that $\Omega_n$ is a closed subset of $\Aut(T)^+$ and that
		\begin{equation*}
		\Omega_n\qq \subseteq \qq h_{n}\qq \Fix_{\Aut(T)^+}(B_T(\gamma,n+\max(Y_1))).
		\end{equation*}
		Since the right hand side is a compact subset of $\Aut(T)^+$ the results follows.
	\end{proof}
	
	\noindent We are finally able to prove the result announced at the beginning of the Section \ref{Section factorization}.
	\begin{theorem}\label{Theorem S 0 est fertile }
		The generic filtration $\mathcal{S}_{0}$ of $G^+_{(i)}(Y_0,Y_1)$ factorizes$^+$ at all depth $l\geq 1$.
	\end{theorem}
	\begin{proof}
		For shortening of the formulation, we let $G=G^+_{(i)}(Y_0,Y_1)$. To prove that $\mathcal{S}_0$ factorizes$^+$ a depth $l\geq 1$, we shall successively verify the three conditions of the Definition \ref{definition olsh facto}.
		
		First, we need to prove that for every  $U$ in the conjugacy class of an element of $\mathcal{S}_0\lb l \rb$ and every subgroup $V$ in the conjugacy class of an element of $\mathcal{S}_0$ with $V \not\subseteq U$, there exists a $W$ in the conjugacy class of an element of $\mathcal{S}_0\lb l -1\rb$ such that $$U\subseteq W \subseteq V U.$$ Let $U$ and $V$ be as above. If $V$ is conjugate to an element of $\mathcal{S}_0\lb l'\rb$ for some $l'\geq l$, the result follows directly from Proposition \ref{Proposition la premiere etape de la fcatorization pour GY1Y_2}. Therefore, let us suppose that $l'<l$. By the definition of $\mathcal{S}_0$ and since $\mathfrak{T}_0$ is stable under the action of $G$, there exist two subtrees $\mathcal{T},\mathcal{T}'\in \mathfrak{T}_0$ such that $U=\Fix_G(\mathcal{T})$ and $V=\Fix_G(\mathcal{T}')$. There are two cases. Either, $\mathcal{T}'\subseteq \mathcal{T}$ and there exists a subtree $\mathcal{P}\in \mathfrak{T}_{0}$ such that $\mathcal{T}'\subseteq \mathcal{P}\subseteq \mathcal{T}$ and $\Fix_G(\mathcal{R})\in \mathcal{S}_0\lb l-1\rb$. In that case $$\Fix_G(\mathcal{T})\subseteq \Fix_G(\mathcal{P})\subseteq \Fix_G(\mathcal{T}')\subseteq \Fix_G(\mathcal{T}')\Fix_G(\mathcal{T})$$ and the result follows. Or else, $\mathcal{T}'\not \subseteq \mathcal{T}$ and since $l'< l$, there exists a subtree $\mathcal{P}\in \mathfrak{T}_0$ such that $\mathcal{T}'\subseteq \mathcal{P}\not= \mathcal{T}$ and $\Fix_G(\mathcal{P})\in \mathcal{S}_0\lb l\rb$. In particular, Proposition \ref{Proposition la premiere etape de la fcatorization pour GY1Y_2} ensures the existence of a subgroup $W\in \mathcal{S}_0\lb l-1\rb$ such that $U\subseteq W \subseteq \Fix_G(\mathcal{P})U\subseteq \Fix_G(\mathcal{T}')U$. This proves the first condition.
		
		Next, we need to prove that $N_{G}(U, V)= \{g\in {G} \lvert g^{-1}Vg\subseteq U\}$ is compact for every $V$ in the conjugacy class of an element of $\mathcal{S}_0$. Just as before, notice that there exists a $\mathcal{T}'\in \mathfrak{T}_0$ such that $V = \Fix_G(\mathcal{T}')$. Since $G$ satisfies the hypothesis \ref{Hypothese H0}, notice that
		\begin{equation*}
		\begin{split}
		N_{G}(U, V)&= \{g\in G \lvert g^{-1}Vg\subseteq U\}=\{g\in G \lvert g^{-1}\Fix_G(\mathcal{T}')g\subseteq \Fix_G(\mathcal{T})\}\\ 
		&=\{g\in G \lvert \Fix_G(g^{-1}\mathcal{T}')\subseteq \Fix_G(\mathcal{T})\}= \{g\in G \lvert g\mathcal{T}\subseteq \mathcal{T}'\}.
		\end{split}
		\end{equation*}
		In particular, since both $\mathcal{T}$ and $\mathcal{T}'$ are finite subtrees of $T$, $N_G(U,V)$ is a compact subset of $G$ which proves the second condition.
		
		Finally, we need to prove for every $W$ in the conjugacy class of an element of $\mathcal{S}_0\lb l-1\rb$ with $U\subseteq W$ that 
		\begin{equation*}
		W\subseteq N_G(U,U) =\{g\in G \mid g^{-1}Ug\subseteq U\}.
		\end{equation*} 
		For the same reasons as before, there exists $\mathcal{R}\in \mathfrak{T}_0$ such that $W= \Fix_G(\mathcal{R})$. Furthermore, since $U\subseteq W$ and since $G$ satisfies the hypothesis \ref{Hypothese H0}, notice that $\mathcal{R}\subseteq \mathcal{T}$. Moreover, since $\Fix_G(\mathcal{R})$ has depth $l-1$, notice that $\mathcal{R}$ contains every interior vertex of $\mathcal{T}$. Since $G$ is unimodular and satisfies the hypothesis \ref{Hypothese H0}, this implies that
		
		\begin{equation*}
		\begin{split}
		\Fix_G(\mathcal{R})&\subseteq \{h\in G \lvert h\mathcal{T}\subseteq \mathcal{T}\}= \{h\in G \lvert \Fix_G(\mathcal{T})\subseteq\Fix_G(h\mathcal{T})\}\\
		&=\{h\in G \lvert h^{-1}\Fix_G(\mathcal{T})h\subseteq\Fix_G(\mathcal{T})\}= N_G(U,U)
		\end{split}
		\end{equation*}
		which proves the third condition.
	\end{proof}

	\subsection{Description of cuspidal representations}\label{section cuspidal rep description}
	
	The purpose of this section is to give a description of the cuspidal representations of $G_{(i)}^+(Y_0,Y_1)$. This is done by Theorem \ref{la classification des reresentations cuspidale} below but requires some preliminaries. We refer to \cite{Semal2021O} for proofs and details about the formalism.
	
	Let $G$ be a non-discrete unimodular totally disconnected locally compact group $G$ and let $\mathcal{S}$ be a generic filtration of $G$ factorizing$^+$ at depth $l$. Then, for every irreducible representation $\pi$ of $G$ at depth $l$, \cite[Theorem A]{Semal2021O} ensures the existence of a unique conjugacy class $C_\pi\in \mathcal{F}_{\mathcal{S}}=\{\mathcal{C}(U)\lvert U\in \mathcal{S}\}$ at height $l$ such that $\pi$ admits non-zero $U$-invariant vectors for any $U\in C_\pi$. The conjugacy class $C_\pi$ is called the \tg{seed} of $\pi$. We define the \textbf{group of automorphisms} $\Aut_{G}(C)$ \textbf{of the seed} $C$ as the quotient $N_G(U)/U$ corresponding to any $U\in C$. This group $\Aut_{G}(C)$ is finite and does not depend up to isomorphism on our choice of $U\in C$. Now, let $p_{U}: N_G(U)\mapsto N_G(U)/U$ denote the quotient map, let
	\begin{equation*}
	\tilde{\mathfrak{H}}_{\mathcal{S}}(U)= \{W \mid \exists g\in G\mbox{ s.t. }gWg^{-1}\in \mathcal{S}\lb l-1 \rb \mbox{ and } U \subseteq  W \}.
	\end{equation*}
	and set $$\mathfrak{H}_{\mathcal{S}}(C)=\{ p_U(W)\lvert W\in \tilde{\mathfrak{H}}_\mathcal{S}(U)\}.$$ Notice that $\mathfrak{H}_{\mathcal{S}}(C)$ does not depend on our choice of representative $U\in C$. 
	\begin{definition}\label{definitoin standard representations}
		An irreducible representation $\omega$ of $\Aut_{G}(C)$ is called $\mathcal{S}$-\tg{standard} if it has no non-zero $H$-invariant vector for any $H\in \mathfrak{H}_{\mathcal{S}}(C)$.
	\end{definition}
	The importance of this notion is given by \cite[Theorem A]{Semal2021O} which ensures that the irreducible representations of $G$ at depth $l$ with seed $C$ are obtained from the $\mathcal{S}$-standard representations of $\Aut_G(C)$ if $\mathcal{S}$ factorizes$^+$ at depth $l$. To be more precise, we recall that every irreducible representation $\omega$ of $\Aut_G(C)\cong N_G(U)/U$ can be lifted to an irreducible representation $\omega \circ p_U$ of $N_G(U)$ acting trivially on $U$ and with representation space $\Hr{\omega}$. The lifted representation can then be induced to $G$. The resulting representation $$T(U,\omega)=\Ind_{N_G(U)}^G(\omega\circ p_{U})$$ is an irreducible representation of $G$ with seed $\mathcal{C}(U)$. Conversely, if $\pi$ is an irreducible representation of $G$ with seed $C$, notice that $\Hr{\pi}^{U}$ is a non-zero $N_G(U)$-invariant subspace of $\Hr{\pi}$ for every $U\in C$. In particular, the restriction $(\restriction{\pi}{N_G(U)}, \Hr{\pi}^U)$ is a representation of $N_G(U)$ whose restriction to $U$ is trivial. This representation passes to the quotient group $N_G(U)/U$ and defines an $\mathcal{S}$-standard representation $\omega_\pi$ of $\Aut_{G}(C)$. 
	
	We now come back to the case we are interested in this paper. Let $T$ be a $(d_0,d_1)$-semi-regular tree with $d_0,d_1\geq 4$, let $V(T)=V_0\sqcup V_1$ be the associated bipartition and let 
	$$\mathfrak{T}_{0}=\{B_T(v,r)\lvert v\in V(T), r\geq 1\}\sqcup\{B_T(e,r)\lvert e\in E(T), r\geq 0\}.$$
	Let $i$ be a legal coloring of $T$, let $Y_0,Y_1\subseteq \N$ be two finite subsets and consider the group $G_{(i)}^+(Y_0,Y_1)$. For shortening of the formulation, let $G=G_{(i)}^+(Y_0,Y_1)$. We have shown in Sections \ref{chapter generic filtration for Gi+YY} and \ref{Section factorization} that
	\begin{equation*}
	\mathcal{S}_{0}=\{\Fix_G(\mathcal{T})\lvert \mathcal{T}\in \mathfrak{T}_{0}\}.
	\end{equation*}
	is a generic filtration of $G_{(i)}^+(Y_0,Y_1)$ that factorizes$^+$ at all depth $l\geq 1$. In particular, \cite[Theorem A]{Semal2021O} provides a bijective correspondence between irreducible representations of $G_{(i)}^+(Y_0,Y_1)$ at depth $l$ with seed $C \in \mathcal{F}_{\mathcal{S}_0}$ and the $\mathcal{S}_0$-standard representations of $\Aut_G(C)$. We start by identifying those seeds and show that the cuspidal representations of $G$ are precisely the irreducible representations of $G$ at depth $l\geq 1$ with respect to $\mathcal{S}_0$. In light of Lemma \ref{Lemme la forme des elements de Sl} we consider the partition $\mathfrak{T}_0=\bigsqcup_{l\in \N}\mathfrak{T}_0\lb l\rb$ where:
	\begin{itemize}\label{de la merde one definit mathfrak T 0}
		\item $\mathfrak{T}_0\lb l\rb = \{B_T(e,\frac{l}{2})\lvert e\in E(T)\}$ if $l$ is even.
		\item $\mathfrak{T}_0\lb l \rb=\{ B_T\big(v,(\frac{l+1}{2})\big)\mid v\in V(T)\}$ if $l$ is odd.
	\end{itemize} 
	Notice that for all $l\in \N$, $\mathfrak{T}_0\lb l\rb$ is stable under the action of $G$. Furthermore, notice if $l$ is even or if $l$ is odd and $G$ is transitive on the vertices that the set $\mathfrak{T}_0\lb l\rb$ consists of a single $G$-orbit. On the other hand, if $l$ is odd and $G$ has two orbits of vertices notice that the set $\mathfrak{T}_0\lb l\rb$ consists of two $G$-orbits namely $\{B_T\big(v,(\frac{l+1}{2})\big)\mid v\in V_0\}$ and $\{B_T\big(v,(\frac{l+1}{2})\big)\mid v\in V_1\}$. In particular, in light of Lemma \ref{lemma G satisfies hypotheses H0}, there are either one or two elements of $\mathcal{F}_{\mathcal{S}_0}=\{\mathcal{C}(U)\lvert U\in \mathcal{S}_0\}$ at height $l$ and each such element is of the form
	$$C=\{\Fix_G(\mathcal{T})\lvert \mathcal{T}\in \mathcal{O}\}$$
	where $\mathcal{O}$ is a $G$-orbit of $\mathfrak{T}_0\lb l \rb$. We deduce easily that the irreducible representations of $G_{(i)}^+(Y_0,Y_1)$ at depth $l\geq 1$ with respect to $\mathcal{S}_0$ are the cuspidal representations of $G$. Indeed, $\pi$ is an irreducible representation at depth $l\geq 1$ with respect to $\mathcal{S}_0$ if and only if $\pi$ does not admit a non-zero $V$-invariant vector for any $V$ in a conjugacy class $C$ at depth $0$ that is for any $V\in \{\Fix_G(e)\lvert e\in E(T)\}$. Now, let $\pi$ be a cuspidal representations of $G$, let $C_\pi\in \mathcal{F}_{\mathcal{S}_0}$ be the seed of $\pi$, let $U\in C_\pi$ and let $\mathcal{T}\in \mathfrak{T}_0$ be such that $U=\Fix_G(\mathcal{T})$. Since $\mathcal{S}_0$ factorizes $^+$ at all depth $l\geq 1$, \cite[Theorem A]{Semal2021O} ensures that $\pi$ is induced from an irreducible representation of $N_G(U)$ that passes to the quotient $\Aut_G(C_\pi)\cong N_G(U)/U$. Furthermore, since $G$ satisfies the hypothesis \ref{Hypothese H0}, notice that 
	\begin{equation*}
	\begin{split}
	N_G(U)= \{g\in G\lvert gUg^{-1}=U\}
	&=\{g\in G\lvert g\Fix_G(\mathcal{T})g^{-1}=\Fix_G(\mathcal{T})\}\\
	&=\{g\in G\lvert \Fix_G(g\mathcal{T})=\Fix_G(\mathcal{T})\}\\
	&=\{g\in G\lvert g\mathcal{T}=\mathcal{T}\}=\Stab_G(\mathcal{T})
	\end{split}
	\end{equation*}
	is exactly the stabilizer of $\mathcal{T}$ in $G_{(i)}^+(Y_0,Y_1)$. In particular, $\Aut_G(C_\pi)$ can be identified with the automorphism group of $\mathcal{T}$ obtained by restricting the action of $\Stab_G(\mathcal{T})$ to $\mathcal{T}$. Moreover, since $G$ satisfies the hypothesis \ref{Hypothese H0} notice that 
	\begin{equation*}
	\tilde{\mathfrak{H}}_{\mathcal{S}_0}(U)=\{\Fix_G(\mathcal{R})\lvert \mathcal{R}\in \mathcal{T}_0, \qq \mathcal{R}\subsetneq \mathcal{T} \mbox{ and }\mathcal{R}\mbox{ is maximal for this property}\}
	\end{equation*}
	and
	$$\mathfrak{H}_{\mathcal{S}_0}(C_\pi)=\{ p_U(W)\lvert W\in \tilde{\mathfrak{H}}_{\mathcal{S}_0}(U)\}$$ 
	is the set of fixators (in $\Aut_G(C_\pi)$) of subtrees $\mathcal{R}\in \mathcal{T}_0$ satisfying $\mathcal{R}\subsetneq \mathcal{T}$ and which are maximal for this property. In particular, the $\mathcal{S}_0$-\tg{standard} representations of $\Aut_G(C_\pi)$ are the irreducible representations of the group of automorphisms of $\mathcal{T}$ obtained by restricting the action of $\Stab_G(\mathcal{T})$ and which do not admit any non-zero invariant vector for the fixator of any subtree $\mathcal{R}$ of $\mathcal{T}$ which belongs to $\mathfrak{T}_0$ and is maximal for this property. The above discussion together with \cite[Theorem A]{Semal2021O}  leads to following description of the cuspidal representations of $G_{(i)}^+(Y_0,Y_1)$.
	
	\begin{theorem}\label{la classification des reresentations cuspidale}
		Let $T$ be a $(d_0,d_1)$-semi-regular tree with $d_0,d_1\geq 4$, let $i$ be a legal coloring of $T$, let $Y_0,Y_1\subseteq \N$ be two finite subsets, let $G=G_{(i)}^+(Y_0,Y_1)$, consider the generic filtration  $\mathcal{S}_0$ of $G$ (defined in Section \ref{chapter generic filtration for Gi+YY}) and let us use the above notations. Then, the cuspidal representations of $G$ are exactly the irreducible representations at depth $l\geq 1$ with respect to $\mathcal{S}_{0}$. Furthermore, if $\pi$ is a cuspidal representation at depth $l$ we have that:
		\begin{itemize}
			\item $\pi$ has no non-zero $\Fix_G(\mathcal{R})$-invariant vector for any $\mathcal{R}\in \bigsqcup_{r<l}\mathfrak{T}_0\lb r\rb$.
			\item There exists a unique conjugacy class $C_\pi\in  \mathcal{F}_{\mathcal{S}_0}$ at height $l$ such that $\pi$ admits a non-zero $U$-invariant vector for any (hence for all) $U\in C_\pi$. Equivalently, there exists a unique $G$-orbit $\mathcal{O}$ of $\mathfrak{T}_0\lb l \rb$  such that $\pi$ admits a non-zero $\Fix_G(\mathcal{T})$-invariant vector for any (hence for all) $\mathcal{T}\in \mathcal{O}$. Furthermore, $\mathcal{O}$ is the only orbit of $\mathfrak{T}_0$ under the action of $G$ such that $C_\pi=\{\Fix_G(\mathcal{T})\lvert \mathcal{T}\in \mathcal{O}\}$.
			\item If $\mathcal{O}$ is the unique $G$-orbit of $\mathfrak{T}_0\lb l \rb$ corresponding to $\pi$ and if $\mathcal{T}\in \mathcal{O}$, $\pi$ admits a non-zero diagonal matrix coefficient supported in $\Stab_G(\mathcal{T})$. In particular, $\pi$ is square-integrable its equivalence class is isolated in the unitary dual $\widehat{G}$ for the Fell topology.	
		\end{itemize}
		Furthermore for every $C\in \mathcal{F}_{\mathcal{S}_0}$ at height $l\geq 1$ with corresponding $G$-orbit $\mathcal{O}$ in $\mathfrak{T}_0\lb l \rb$ that is $C=\{\Fix_G(\mathcal{T})\lvert \mathcal{T}\in \mathcal{O}\}$, there exists a bijective correspondence between the equivalence classes of irreducible representations of $G$ with seed $C$ and the equivalence classes of $\mathcal{S}_0$-standard representations of $\Aut_G(C)$. More precisely, for every $\mathcal{T}\in \mathcal{O}$ the following holds:
		\begin{enumerate}
			\item If $\pi$ is a cuspidal representation of $G$ with seed $C$, $(\omega_\pi, \Hr{\pi}^{\Fix_G(\mathcal{T})})$ is an $\mathcal{S}_0$-standard representation of $\Aut_{G}(C)$ such that
			\begin{equation*}
			\pi \cong T(\Fix_G(\mathcal{T}),\omega_\pi)=\Ind_{\Stab_G(\mathcal{T})}^{G}(\omega_\pi \circ p_{\Fix_G(\mathcal{T})}).
			\end{equation*}
			\item If $\omega$ is an $\mathcal{S}_0$-standard representation of $\Aut_{G}(C)$, the representation $T(\Fix_G(\mathcal{T}),\omega)$ is a cuspidal of $G$ with seed in $C$. 
		\end{enumerate}
		Furthermore, if $\omega_1$ and $\omega_2$ are $\mathcal{S}_0$-standard representations of $\Aut_{G}(C)$, we have that $T(\Fix_G(\mathcal{T}),\omega_1)\cong T(\Fix_G(\mathcal{T}),\omega_2)$ if and only if $\omega_1\cong\omega_2$. In particular, the above two constructions are inverse of one an other.
	\end{theorem}

	\subsection{Existence of cuspidal representations}\label{section existence des cuspidale}
	
	Just as in the above sections, let $T$ be a $(d_0,d_1)$-semi-regular tree with $d_0,d_1\geq 4$, let $V(T)=V_0\sqcup V_1$ be the associated bipartition and let 
	$$\mathfrak{T}_{0}=\{B_T(v,r)\lvert v\in V(T), r\geq 1\}\sqcup\{B_T(e,r)\lvert e\in E(T), r\geq 0\}.$$
	Let $i$ be a legal coloring of $T$, let $Y_0,Y_1\subseteq \N$ be two finite subsets and consider the group $G_{(i)}^+(Y_0,Y_1)$. For shortening of the formulation and when it leads to no confusion we let $G=G_{(i)}^+(Y_0,Y_1)$. We have shown in Sections \ref{chapter generic filtration for Gi+YY} and \ref{Section factorization} that
	\begin{equation*}
	\mathcal{S}_{0}=\{\Fix_G(\mathcal{T})\lvert \mathcal{T}\in \mathfrak{T}_{0}\}.
	\end{equation*}
	is a generic filtration of $G_{(i)}^+(Y_0,Y_1)$ that factorizes$^+$ at all depth $l\geq 1$. In particular, \cite[Theorem A]{Semal2021O} provided a bijective correspondence between the irreducible representations of $G_{(i)}^+(Y_0,Y_1)$ at depth $l$ with seed $C \in \mathcal{F}_{\mathcal{S}_0}$ and the $\mathcal{S}_0$-standard representations of $\Aut_G(C)$. This lead to a description of the cuspidal representations of $G_{(i)}^+(Y_0,Y_1)$ see Theorem \ref{la classification des reresentations cuspidale}. On the other hand, none of those results yet ensures the existence of a cuspidal representation of $G_{(i)}^+(Y_0,Y_1)$. The purpose of this section is to prove the existence of a cuspidal representation with seed $C$ for each conjugacy class $C\in \mathcal{F}_{\mathcal{S}_0}$ at height $l\geq 1$. From the description of cuspidal representations of $G^+_{(i)}(Y_0,Y_1)$ provided by Theorem \ref{la classification des reresentations cuspidale}, it is equivalent to prove the following theorem.
	\begin{theorem}
		Let $G=G^+_{(i)}(Y_0,Y_1)$ and let $C\in \mathcal{F}_{\mathcal{S}_0}$ be a conjugacy class at height $l\geq 1$. Then, there exists a ${\mathcal{S}_0}$-standard representations of $\Aut_G(C)$. 
	\end{theorem}
	The proof of this theorem is gathered in the few results below. We start by recalling a result from \cite{Semal2021O}. 
	\begin{proposition}[{\cite[Proposition $2.29$]{Semal2021O}}]\label{existence criterion}
		Let $T$ be a locally finite tree, let $G\leq \Aut(T)$ be a closed subgroup, let $\mathcal{T}$ be a finite subtree of $T$ and let $\{\mathcal{T}_1,\mathcal{T}_2,...,\mathcal{T}_s\}$ be a set of distinct finite subtrees of $T$ contained in $\mathcal{T}$ such that $\mathcal{T}_i\cup\mathcal{T}_j=\mathcal{T}$ for every $i\not=j$. Suppose that $\Stab_G(\mathcal{T})$ acts by permutation on the set $\{\mathcal{T}_1,\mathcal{T}_2,...,\mathcal{T}_s\}$ and that $\Fix_G(\mathcal{T})\subsetneq \Fix_G(\mathcal{T}_i)\subsetneq \Stab_G(\mathcal{T})$. Then, there exists an irreducible representation of $\Stab_G(\mathcal{T})/\Fix_G(\mathcal{T})$ without non-zero $\Fix_G(\mathcal{T}_i)/\Fix_G(\mathcal{T})$-invariant vector for every $i=1,...,s$.
	\end{proposition}
	The following proposition ensures the existence of ${\mathcal{S}_0}$-standard representation of $\Aut_{G}(C)$ for all $C\in \mathcal{F}_{\mathcal{S}_0}$ with even height $l\geq 1$.
	\begin{proposition}\label{Proposition existence of standard for edges }
		Let $G=G^+_{(i)}(Y_0,Y_1)$ and let $C\in \mathcal{F}_{\mathcal{S}_0}$ be a conjugacy class at even height $l\geq 1$. Then, there exists a ${\mathcal{S}_0}$-standard representation of $\Aut_{G}(C)$.
	\end{proposition} 
	\begin{proof}
		Since $l$ is even, Lemma \ref{Lemme la forme des elements de Sl} ensures the existence of an edge $e\in E(T)$ and an integer $r\geq 1$ such that $B_T(e,r)\in \mathfrak{T}_0$ and $C=\mathcal{C}(\Fix_{G}(B_T(e,r))$. For shortening of the formulation, we let $\mathcal{T}=B_T(e,r)$. As observed in Section \ref{section cuspidal rep description} we have that $\Aut_G(C)\cong \Stab_G(\mathcal{T})/\Fix_G(\mathcal{T})$ and
		\begin{equation*}
		\begin{split}
		\mathfrak{H}_{\mathcal{S}_0}(\Fix_G(\mathcal{T}))&=\big\{\Fix_G(\mathcal{R})/\Fix_G(\mathcal{T})\lvert \mathcal{R}\in \mathfrak{T}_0, \qq \mathcal{R}\subsetneq \mathcal{T} \\
		&\mbox{ }\mbox{ }\mbox{ }\mbox{ }\mbox{ }\mbox{ }\mbox{ }\mbox{ }\mbox{ }\mbox{ }\mbox{ }\mbox{ }\mbox{ }\mbox{ }\mbox{ }\mbox{ }\mbox{ }\mbox{ }\mbox{ }\mbox{ }\mbox{ }\mbox{ }\mbox{ }\mbox{ }\mbox{ }\mbox{ }\mbox{ }\mbox{ }\mbox{ }\mbox{ }\mbox{ }\mbox{ }\mbox{ }\mbox{ }\mbox{ and } \mathcal{R}\mbox{ is maximal for this property}\big\}\\
		&=\{\Fix_G(B_T(v,r))/\Fix_G(B_T(e,r))\lvert v\in e\}.
		\end{split}
		\end{equation*} 
		Let $v_0,v_1$ be the two vertices of $e$, set $\mathcal{T}_i=B_T(v_i,r)$ and notice that $\mathcal{T}_0\cup \mathcal{T}_1=\mathcal{T}$. Moreover, notice that $\Stab_G(\mathcal{T})=\{g\in G \lvert ge=e\}$ acts by permutations on the set $\{\mathcal{T}_0, \mathcal{T}_1\}$. Furthermore, since $G$ satisfies the hypothesis \ref{Hypothese H0} notice that $\Fix_G(\mathcal{T})\subsetneq \Fix_G(\mathcal{T}_i)\subsetneq \Stab_G(\mathcal{T})$. In particular, the hypotheses of Proposition \ref{existence criterion} are satisfied and the result follows.
	\end{proof}
	A similar reasoning leads to a proof of the existence of ${\mathcal{S}_0}$-standard representation of $\Aut_{G}(C)$ for all $C\in \mathcal{F}_{\mathcal{S}_0}$ with odd height $l> 1$.
	\begin{lemma}\label{existence of standard for odd depth}
		Let $G=G^+_{(i)}(Y_0,Y_1)$ and let $C\in \mathcal{F}_{\mathcal{S}_0}$ be a conjugacy class at odd height $l>1$. Then, there exists a ${\mathcal{S}_0}$-standard representation of $\Aut_{G}(C)$.
	\end{lemma}
	\begin{proof}
		Since $l$ is even, Lemma \ref{Lemme la forme des elements de Sl} ensures the existence of a vertex $v\in V(T)$ and an integer $r\geq 1$ such that $B_T(v,r+1)\in \mathfrak{T}_0$ and $C=\mathcal{C}(\Fix_{G}(B_T(v,r+1))$. For shortening of the formulation, we let $\mathcal{T}=B_T(v,r+1)$ of $T$.  As observed in Section \ref{section cuspidal rep description} we have that $\Aut_G(C)\cong \Stab_G(\mathcal{T})/\Fix_G(\mathcal{T})$ and
		\begin{equation*}
		\begin{split}
		\mathfrak{H}_{\mathcal{S}_0}(\Fix_G(\mathcal{T}))&=\big\{\Fix_G(\mathcal{R})/\Fix_G(\mathcal{T})\lvert \mathcal{R}\in \mathfrak{T}_0, \qq \mathcal{R}\subsetneq \mathcal{T} \\
		&\mbox{ }\mbox{ }\mbox{ }\mbox{ }\mbox{ }\mbox{ }\mbox{ }\mbox{ }\mbox{ }\mbox{ }\mbox{ }\mbox{ }\mbox{ }\mbox{ }\mbox{ }\mbox{ }\mbox{ }\mbox{ }\mbox{ }\mbox{ }\mbox{ }\mbox{ }\mbox{ }\mbox{ }\mbox{ }\mbox{ }\mbox{ }\mbox{ }\mbox{ }\mbox{ }\mbox{ }\mbox{ }\mbox{ }\mbox{ }\mbox{ and } \mathcal{R}\mbox{ is maximal for this property}\big\}\\
		&=\{\Fix_G(B_T(e,r-1))/\Fix_G(B_T(v,r))\lvert e\in E(B_T(v,1))\}.
		\end{split}
		\end{equation*} 
		Let $\{w_1,...,w_d\}$ be the leaves of $B_T(v,1)$. For every $i=1,...,d$ let $\mathcal{T}_i= (B_T(v,1)\backslash\{w_i\})^{(r-1)}$ where $\mathcal{R}^{(t)}=\{v\in V(T)\lvert \exists w\in \mathcal{R} \mbox{ s.t. }d_T(v,w)\leq t\}$. Notice that $\mathcal{T}_i\cup \mathcal{T}_j=\mathcal{T}$. On the other hand, in our case $$\Stab_G(\mathcal{T})=\{g\in G \lvert gv=v\}=\Fix_G(v)$$ and $\Fix_G(v)$ acts by permutation on the set $\{\mathcal{T}_1,...,\mathcal{T}_d\}$. Finally, notice that $v\in \mathcal{T}_i$ and therefore that $\Fix_G(\mathcal{T}_i)\subseteq \Stab_G(\mathcal{T})$ for all $i=1,...,d$. Furthermore, since $G$ contains $G^{+}_{(i)}(\{0\},\{0\})$, notice that $\Fix_G(\mathcal{T})\subsetneq \Fix_G(\mathcal{T}_i)\subsetneq \Stab_G(\mathcal{T})$. In particular, \cite[Proposition 2.27]{Semal2021O} ensures the existence of an irreducible representation of $\Aut_{G}(C)$ without non-zero $\Fix_G(\mathcal{T}_i)/\Fix_G(\mathcal{T})$-invariant vectors. The result follows from the fact that for every edge $e\in E(B_T(v,1))$ there exists some $i\in \{1,...,d\}$ such that $B_T(e,r)\subseteq \mathcal{T}_i$ and hence $\Fix_G(B_T(e,r))/\Fix_G(\mathcal{T})\subseteq \Fix_G(\mathcal{T}_i)/\Fix_G(\mathcal{T})$.
	\end{proof}
	The next lemma treats the remaining case $l=1$ where Proposition \ref{existence criterion} does not apply.
	\begin{lemma}\label{existence of standard for vertices}
		Let $G=G^+_{(i)}(Y_0,Y_1)$ and let $C\in \mathcal{F}_{\mathcal{S}_0}$ be a conjugacy class at height $1$. Then, there exists a $\mathcal{S}_{0}$-standard representation of $\Aut_{G}(C)$.
	\end{lemma}
	\begin{proof}
		Lemma \ref{Lemme la forme des elements de Sl} ensures the existence of a vertex $v\in V(T)$ such that $C=\mathcal{C}(\Fix_{G}(B_T(v,1))$. For shortening of the formulation, we let $\mathcal{T}=B_T(v,r)$. We recall as observed in Section \ref{section cuspidal rep description} that $\Aut_G(C)\cong \Stab_G(\mathcal{T})/\Fix_G(\mathcal{T})$ where $\Stab_G(\mathcal{T})=\{g\in G \mid g\mathcal{T}\subseteq \mathcal{T}\}= \{ g\in G \lvert gv=v\}=\Fix_G(v)$. In particular, $\Aut_G(C)$ can be realised as a the group of automorphisms of $B_T(v,1)$ obtained by restricting the action of $\Fix_G(v)$. Furthermore, we have that  
		\begin{equation*}
		\begin{split}
		\mathfrak{H}_{\mathcal{S}_0}(\Fix_G(\mathcal{T}))&=\big\{\Fix_G(\mathcal{R})/\Fix_G(\mathcal{T})\lvert \mathcal{R}\in \mathfrak{T}_0, \qq \mathcal{R}\subsetneq \mathcal{T} \\
		&\mbox{ }\mbox{ }\mbox{ }\mbox{ }\mbox{ }\mbox{ }\mbox{ }\mbox{ }\mbox{ }\mbox{ }\mbox{ }\mbox{ }\mbox{ }\mbox{ }\mbox{ }\mbox{ }\mbox{ }\mbox{ }\mbox{ }\mbox{ }\mbox{ }\mbox{ }\mbox{ }\mbox{ }\mbox{ }\mbox{ }\mbox{ }\mbox{ }\mbox{ }\mbox{ }\mbox{ }\mbox{ }\mbox{ }\mbox{ }\mbox{ and } \mathcal{R}\mbox{ is maximal for this property}\big\}\\
		&=\{\Fix_G(f)/\Fix_G(B_T(v,1))\lvert f\in E(B_T(v,1))\}.
		\end{split}
		\end{equation*} 
		Let $d$ be the degree of $v$ in $T$, let $X=E(B_T(v,1))$ and let $e\in X$. Since $G$ is closed subgroup of $\Aut(T)$ acting $2$-transitively on the boundary $\partial T$ \cite[Lemma 3.1.1]{BurgerMozes2000} ensures that $\underline{G}(v)$ is $2$-transitive. In particular, $\Aut_{G}(C)$ is $2$-transitive $X$ and  \cite[Lemma 2.28]{Semal2021O} ensures the existence of an irreducible representation $\sigma$ of $\Aut_{G}(C)$ without non-zero $\Fix_{\Aut_{G}(C)}(e)$-invariant vector. Since $\Fix_G(v)$ is transitive on $E(B_T(v,1))$, this representation does not admit a non-zero $\Fix_{\Aut_{G}(C)}(f)$-invariant vector for any $f\in E(B_T(v,1))$. The lemma follows from the fact that $\Fix_{\Aut_{G}(C)}(f)=\Fix_G(f)/\Fix_G(B_T(v,1))$.
	\end{proof}

	\section{Simple Radu groups are CCR}\label{section radu groups are Type I}
	
		Let $T$ be a $(d_0,d_1)$-semi-regular tree with $d_0,d_1\geq 4$. Let $i$ be a legal coloring of $T$ and let $Y_0,Y_1\subseteq \N$ be two finite subsets. The purpose of this section is to exploit the classification of the irreducible representations of $G^+_{(i)}(Y_0,Y_1)$ obtained in Section \ref{Section spherical and special rep} and Section \ref{section cuspidal rep description} to prove that $G^+_{(i)}(Y_0,Y_1)$ is uniformly admissible and hence CCR.
		
		We recall that a totally disconnected locally compact group $G$ is uniformly admissible if for every compact open subgroup $K$, there exists a positive integer $k_K$ such that $\dim(\Hr{\pi}^K)<k_K$ for all irreducible representation $\pi$ of $G$. In particular, uniformly admissible groups are CCR. The following classical result ensures that the spherical representations of $G$ are uniformly admissible.
		%On the other hand, as the following classical result 
		%\begin{lemma}\label{equivalence de l'admissibiliter}
			%Let $G$ be a totally disconnected locally compact group. Then, the following are equivalent :
			%\begin{itemize}
				%\item $\forall \pi \in \widehat{G}$ and for every $K\leq G$ compact open we have $\dim(\Hr{\pi}^K)<+\infty$.
				%\item $\forall \pi \in \widehat{G}$, for every $K\leq G$ compact open and for every $\sigma\in \widehat{K}$, we have that the multiplicity of $\sigma$ in %the direct sum decomposition of $\restr{\pi}{K}$ is finite.
		%	\end{itemize}
		%\end{lemma}
		%\begin{proof}
		%	It is clear that $2$ implies $1$. To prove that the reverse holds, let us suppose that for every $ \pi \in \widehat{G}$ and for every $\tilde{K}\leq G$ compact open we have that $\dim(\Hr{\pi}^{\tilde{K}})<+\infty$. Now, let $K\leq G$ be a compact open subgroup of $G$ and let $\sigma\in \widehat{K}$. Since $K$ is a compact totally disconnected locally compact group, notice that the kernel of $\sigma$ is an open subgroup of $K$. Hence, $\mbox{Ker}(\sigma)$ is a compact open subgroup of $G$. Furthermore, notice that the multiplicity $n_\sigma$ of $\sigma$ in $\restr{\pi}{K}$ is bounded above by $\dim(\Hr{\pi}^{\mbox{Ker}(\sigma)})$ which is finite by hypothesis. This proves $2$.
		%\end{proof}
		%\begin{definition}
		%	An irreducible representation $\pi$ of a totally disconnected locally compact group is called \textbf{admissible} if it satisfies any of the conditions of Lemma \ref{equivalence de l'admissibiliter}. 
		%\end{definition} 
	
	\begin{theorem}\label{Theorem les representations spheriques sont admissibles}
		Let $T$ be a $(d_0,d_1)$-semi-regular tree with $d_0,d_1\geq 3$, let $G\leq \Aut(T)$ be a closed non-compact subgroup acting transitively on the boundary of $T$ and let $v\in V(T)$. Then, for every integer $n\geq 1$ there exists a constant $k_n\in \N$ such that ${\rm dim}(\Hr{\pi}^{K_n})<k_n$ for every spherical representation $\pi$ of $G$ admitting a non-zero $\Fix_G(v)$-invariant vector and where $K_n=\Fix_G(B_T(v,n))$ .
	\end{theorem}
	\begin{proof}
		Let $K=\Fix_G(v)$ and let $\mu$ be the Haar measure of $G$ renormalised in such a way that $\mu(K)=1$. Theorems \ref{thm la classification des spherique cas trans} and \ref{thm la classification des spherique cas 2 orbites} ensure that $(G,K)$ is a Gelfand pair and we observe that ${\rm dim}(\Hr{\pi}^K)= 1.$ Now, let $\xi$ be a unit vector of $\Hr{\pi}^K$ and let $\eta$ be a unit vector of $\Hr{\pi}^{K_n}.$ Notice that 
		$$\varphi_{\xi, \eta}\fct{G}{\C}{g}{\prods{\pi(g)\xi}{\eta}}$$
		is $K$-right invariant and $K_n$-left invariant continuous function. On the other hand, since ${\rm dim}(\Hr{\pi}^K)= 1$, notice for all $g,h\in G$ that 
		\begin{equation*}
		\begin{split}
		\int_K\varphi_{\xi,\eta}(gkh)\qq\diff\mu(k)&=\prods{\int_K \pi(gkh)\xi}{\eta}\\
		&=\prods{\pi(h)\xi}{\int_K\pi(k^{-1})\pi(g^{-1})\eta}\\
		&=\prods{\pi(h)\xi}{\alpha(\eta, g)\xi}=\overline{\alpha(\eta, g)}\varphi_{\xi,\xi}(h)
		\end{split}
		\end{equation*}
		for some $\alpha(\eta, g)\in \C$. However, $\varphi_{\xi,\xi}(1_G)=1$ and hence $\overline{\alpha(\eta, g)}=\varphi_{\xi,\eta}(g)$. This implies for all $g,h\in G$ that
		\begin{equation}\label{equation de semi sphericite}
		\int_K\varphi_{\xi,\eta}(gkh)\qq \diff \mu(k)=\varphi_{\xi,\eta}(g)\varphi_{\xi,\xi }(h).
		\end{equation}
		
		Since $\varphi_{\xi,\eta}$ is $K$-right invariant and $K_n$-left invariant notice that it can be realised as a function $\phi:Gv\rightarrow \C$ on the orbit $Gv$ of $v$ in $V(T)$ that is constant on the $K_n$-orbits of $v$. On the other hand, since $K_n$ is an open subgroup of the compact group $K$ and since $K$, the index of $K_n$ in $K$ is finite. Since $K$ is transitive on the boundary of the tree, this implies that $K_n$ has finitely many orbit on $\partial T$. In particular, there exists an integer $N_n\geq \max\{2,n\}$ such that $\partial T(w,v)$ is contained in a single $K_n$-orbit for all $w\in \partial B_T(v,N_n)$ where $\partial T(w,v)$ is the set of ends of $T(w,v)=\{u\in V(T)\lvert d_T(u,w)<d_T(u,v)\}$ which are not vertices.
		Now, let $t\in G$ be such that $d_T(v,t v)=2$. Notice that equality \eqref{equation de semi sphericite} ensures that the sum of values of $\phi$ on the vertices at distance $2$ from $gv$ is equal to $\phi(gv)\varphi_{\xi,\xi}(t)$. In particular, this implies that $\phi$ is determined entirely by the values it takes in $B_T(v,N_n)$. Hence, the space $\mathcal{L}_n$ of function $\varphi:G\rightarrow \C$ which are $K$-right invariant, $K_n$-left invariant and satisfy the equality \eqref{equation de semi sphericite} has finite dimension bounded by the cardinality $k_n$ of $B_T(v,N_n)$. Furthermore, since $\pi$ is irreducible notice that $\xi$ is cyclic and therefore that the linear map $\Psi_n: \Hr{\pi}^{K_n}\rightarrow \mathcal{L}_n: \eta \rightarrow \overline{\varphi_{\xi,\eta}}$ is injective. This proves as desired that $\dim(\Hr{\pi}^{K_n})\leq \dim(\mathcal{L}_n)\leq k_n<+\infty$.
	\end{proof}
	This allows one to prove what we announced at the beginning of the section.
	\begin{theorem}\label{theorem simple radu groups are unifmroly admissible}
		$G^+_{(i)}(Y_0,Y_1)$ is uniformly admissible, hence CCR.
	\end{theorem}
	\begin{proof}
		For shortening of the formulation, let $G=G^+_{(i)}(Y_0,Y_1)$. Let $K\leq G$ be a compact open subgroup, let $v\in V(T)$ and let $K_n=\Fix_G(B_T(v,n))$ for every $n\in \N$.  Since $(K_n)_{n\in N}$ is a basis of neighbourhood of the identity there exists some $N\in \N$ such that $K_N\subseteq K$. In particular, Theorem \ref{Theorem les representations spheriques sont admissibles} ensures the existence of a positive integer $k^1_K$ such that $\dim(\Hr{\pi}^K)\leq k^1_K$ for every spherical representation $\pi$ of $G$. On the other hand, \cite[Corollary of Theorem 2]{Harish1970} ensures the existence of a positive integer $k^2_K$ such that $\dim(\Hr{\pi}^K)\leq k^2_K$ for every irreducible square-integrable representation $\pi$ of $G$ and Theorems \ref{thm la classification des speciale} and \ref{la classification des reresentations cuspidale} ensures that the special and cuspidal representations are square-integrable. The result therefore follows from the fact that each irreducible $\pi$ of $G$ is either spherical, special or cuspidal.
	\end{proof}
	\newpage
	
	\begin{appendices}
		
	\section{Irreducible representations of a group and subgroups of index $2$}\label{section irreducvtible pour des groupes d'indice 2}
	
	The purpose of this appendix is to explicit the relations between the irreducible unitary representations of a locally compact group $G$ and the irreducible representations of its closed subgroups $H\leq G$ of index $2$. Among other things, Theorem \ref{les rep dun group loc compact par rapport a celle d'un de ses sous groupes} below explicit the correspondence between those representations. Furthermore, when $G$ is a totally disconnected locally compact group, we show that $G$ is uniformly admissible if and only if $H$ is uniformly admissible see Lemma \ref{lemma G admissible iff  H is admissible}.
	
	The relevance of this appendix is given by Theorem \ref{Corollary Radu simple then Radu} which ensures that every Radu group $G$ lies in a finite chain $H_n\geq... \geq H_0$ with $n\in \{0,1,2,3\}$ such that $H_n=G$, $\lb H_t: H_{t-1} \rb =2$ for all $t$ and where $H_0$ is conjugate in $\Aut(T)^+$ to $G_{(i)}^+(Y_0,Y_1)$ if $T$ is a $(d_0,d_1)$-semi-regular tree with $d_0,d_1\geq 6$. As a direct consequence, we therefore obtain a description of the irreducible representations of those groups and observe that they are uniformly admissible. To be more precise, the spherical and special representations of any Radu group $G$ are classified by Section \ref{Section spherical and special rep} and a description of the cuspidal representations of those groups can be obtained from the description of cuspidal representations of $G_{(i)}^+(Y_0,Y_1)$ given in Sections \ref{Section classification des rep de simple radu groups} by applying Theorem \ref{les rep dun group loc compact par rapport a celle d'un de ses sous groupes} $n$ times (where $n$ is as above).
	
	\subsection{Preliminaries}\label{Section appendice reminder}
	Let $G$ be a locally compact group and let $H\leq G$ be a closed normal subgroup of finite index (in particular $H$ is open in $G$). The purpose of this section is to recall three operations that can be applied either to the representations of $G$ or to those of $H$.
	
	We start with the conjugation of representations. For every $g\in G$, we denote by $$c(g): G\rightarrow G: h \mapsto ghg^{-1}$$
	the conjugation map and for every representation $\pi$ of $H$ we define the morphism $\pi^g= \pi \circ c(g)$. Since $H$ is normal in $G$, notice that $\pi^g$ is a well defined representation of $H$ on the Hilbert space $\Hr{\pi}$. This representation is called the \tg{conjugate representation} of $\pi$ by $g$. Furthermore, notice that the conjugate representation $\pi^g$ depends up to equivalence only on the coset $gH$ and that $\pi^g$ is irreducible if and only if $\pi$ is irreducible. In particular, the action by conjugation of $G$ on $\widehat{H}$ pass to the quotient $G/H$.
	
	Now, let $G'$ be an other topological group, let  $\phi:G \rightarrow G'$ be a continuous group homomorphism, let $\pi$ be a representations of $G$ and let $\chi$ be a unitary character of $G'$. We define the \tg{twisted representation} $\pi^\chi$ as the representation of $G$ on $\Hr{\pi}$ given by 
	\begin{equation*}
	\pi^\chi (g)=\chi(\phi(g)) \pi(g)\q \forall g\in G.
	\end{equation*}
	Notice that this representation is still continuous and unitary since $\chi, \pi$ and $\phi$ are continuous group homomorphism and since $\chi$ and $\pi$ are unitary. Furthermore, notice that $\pi^\chi\cong \pi$ if $\chi$ is the trivial representation of $G'$. 
	\begin{lemma}
		$\pi^\chi$ is irreducible if and only if $\pi$ is irreducible.
	\end{lemma}
	\begin{proof}
		Since $\chi(g)$ is a unitary complex number for every $g\in G$, notice that, for every $\xi \in \Hr{\pi}$, the subspace of $\Hr{\pi}$ spanned by $\{\pi(g)\xi\mid g\in G\}$ is the same than the subspace spanned by $\{\chi(g)\pi(g)\xi\mid g\in G\}=\{\pi^\chi (g)\mid g\in G\}$. The result therefore follows from the fact that a representation is irreducible if and only if every non-zero vector is cyclic.
	\end{proof}
	
	Finally, we recall the notion of \textbf{induction}. Since most of the complexity vanishes when $H$ is an open subgroup of $G$ (because the quotient space $G/H$ is discrete) and since it is the only set up encountered in this notes, we will work under this hypothesis. We refer to \cite[Chapters $2.1$ and $2.2$]{KaniuthTaylor2013} for details. Let $G$ be a locally compact group, let $H\leq G$ be an open subgroup and let $\sigma$ be a representation of $H$. The induced representation $\Ind_{H}^G(\sigma)$ is a representation of $G$ with representation space given by
	$$\Ind_H^G(\Hr{\sigma})=\bigg\{\phi:G\rightarrow\Hr{\sigma}\Big\lvert  \phi(gh)=\sigma(h^{-1})\phi(g), \s{gH\in G/H}{}\prods{\phi(g)}{\phi(g)}<+\infty\qq \bigg\}.$$
	For $\psi, \phi\in \Ind_H^G(\Hr{\sigma})$, we let
	\begin{equation*}
	\prods{\psi}{\phi}_{\Ind_H^G(\Hr{\sigma})}= \s{gH\in G/H}{} \prods{\psi(g)}{\phi(g)}.
	\end{equation*}
	Equipped with this inner product, $\Ind_H^G(\Hr{\sigma})$ is a separable complex Hilbert space. The induced representation $\Ind_H^G(\sigma)$ is the representation of $G$ on $\Ind_H^G(\Hr{\sigma})$ defined by
	\begin{equation*}
	\big\lb\Ind_H^G(\sigma)(h)\big\rb\phi(g)=\phi(h^{-1}g)\q \forall \phi \in \Ind_H^G(\Hr{\sigma}) \mbox{ and }\forall g,h\in G.
	\end{equation*}

	\subsection{The explicit correspondence}
	
	We come back to the context we are interested in. Let $G$ be a locally compact group and let $H\leq G$ be a closed subgroup of index $2$ in $G$. In particular, $H$ is an normal open subgroup of $G$ and the quotient $G'=G/H$ is isomorphic to the cyclic group of order two. We let $\tau$ denote the only irreducible non-trivial representation of $G/H$. Let $t\in G-H$, let $\phi: G \rightarrow G/H$ be the canonical projection on the quotient and, for every representation $\pi$ of $G$, let $\pi^\tau$ be the twisted representation of $G$ as defined in the Section \ref{Section appendice reminder}. The purpose of this section is to prove the following theorem.
	\begin{theorem}\label{les rep dun group loc compact par rapport a celle d'un de ses sous groupes}
		For every irreducible representation $\pi$ of $G$ we have that:
		\begin{itemize}
			\item $\pi\not\cong \pi^\tau$ if and only if $\Res_H^G(\pi)$ is an irreducible representation of $H$ and in that case $\Res_H^G(\pi)\cong\Res_H^G(\pi)^t$.
			\item $\pi\cong \pi^\tau$ if and only if $\Res_H^G(\pi)\cong \sigma \oplus \sigma^t$ for some irreducible representation $\sigma$ of $H$ and in that case $\sigma\not \cong \sigma^t$.
		\end{itemize}
		For every irreducible representation $\sigma$ of $H$, we have that:
		\begin{itemize}
			\item $\sigma\not\cong \sigma^t$ if and only if $\Ind_H^G(\sigma)$ is an irreducible representation of $G$ and in that case $\Ind_H^G(\sigma)\cong\Ind_H^G(\sigma)^\tau$.
			\item $\sigma \cong \sigma^t$ if and only if $\Ind_H^G(\sigma)\cong \pi \oplus \pi^\tau$ for some irreducible representation $\pi$ of $G$ and in that case $\pi \not \cong \pi^\tau$.
		\end{itemize}
		Furthermore:
		\begin{enumerate}
			\item Every irreducible representation $\pi$ of $G$ satisfies $\pi\leq \Ind_H^G(\sigma)$ for some irreducible representation $\sigma$ of $H$. 
			\item Every irreducible representation $\sigma$ of $H$ satisfies $\sigma\leq \Res_H^G(\pi)$ for some irreducible representation $\pi$ of $G$.
		\end{enumerate} 
	\end{theorem}
	The proof of this theorem is gathered in the following few results. 
	\begin{lemma}\label{Lemma les rep de G se decompose ou pas en irred si H est d'indice 2}
		Let $\pi$ be an irreducible representation of $G$. Then exactly one of the following happens:
		\begin{itemize}
			\item $\Res_H^G(\pi)$ is an irreducible representation of $H$ and $\Res_H^G(\pi)\cong \Res_H^G(\pi)^t$.
			\item $\Res_H^G(\pi)\cong\sigma\oplus \sigma^t$ for some irreducible representation $\sigma$ of $H$.
		\end{itemize}
	\end{lemma} 
	\begin{proof}
		If $\Res_H^G(\pi)$ is an irreducible representation of $H$, notice for every $h\in H$ and every $\xi \in \Hr{\pi}$ that
		\begin{equation*}
		\begin{split}
			\pi(t)\lb \Res_H^G(\pi)(h)\rb \xi&=\pi(t)\pi(h)\xi=\pi(tht^{-1})\pi(t)\xi\\
			&=\lb \Res_H^G(\pi)(tht^{-1})\rb\pi(t)\xi=\lb\Res_H^G(\pi)^t(h)\rb\pi(t)\xi.
		\end{split}			
		\end{equation*}
		In particular, $\pi(t):\Hr{\pi}\rightarrow\Hr{\pi}$ is an intertwining operator between $\Res_H^G(\pi)$ and $\Res_H^G(\pi)^t$ which settles the first case.
		
		Now, suppose that $\Res_H^G(\pi)$ is not an irreducible representation of $G$. Since $\pi$ is irreducible, any non-zero $\xi\in \Hr{\pi}$ is a cyclic vector. Hence, the subspace spanned by $\{\pi(g)\xi \mid g\in G\}$ is dense in $\Hr{\pi}$. On the other hand, since $\Res_H^G(\pi)$ is not irreducible, there exists a non-zero vector $\xi \in \Hr{\pi}$ such that the subspace spanned by $\{\pi(h)\xi \mid h\in H\}$ is not dense in $\Hr{\pi}$. Let $M$ denote the closure of this space. First, let us show that $(\Res_H^G(\pi),M)$ is irreducible. To this end, let $N$ be a $\pi(H)$-invariant subspace of $M$ and let us show that $N$ can not be a proper subspace of $M$. Notice that for every  closed $\pi(H)$-invariant subspace $L$ of $\Hr{\pi}$, the subspace $\pi(t)L$ is also $\pi(H)$-invariant since $\pi(H)\pi(t)L=\pi(Ht)L=\pi(tH)L=\pi(t)L$. Now, since $\xi$ is a cyclic vector for $\pi$, notice that $\Hr{\pi}=M+ \pi(t)M$ (where the sum is a priori not a direct sum). On the other hand, since $\Hr{\pi }\not=M$ and since $t^2\in H$, notice that $M\not \subseteq \pi(t)M$. In particular, replacing $N$ by $N^\perp\cap M$ if needed, we can suppose that $\Hr{\pi}\not=N +\pi(t) M$ and therefore that $\Hr{\pi}\not= N+\pi(t)N $. On the other hand, $N+ \pi(t)N$ is a closed $\pi(G)$-invariant subspace of $\pi$. Since $\pi$ is irreducible, this implies that $N+\pi(t)N=\{0\}$ and therefore that $N=\{0\}$ which proves that $(\Res_H^G(\pi),M)$ is irreducible. Since $\pi(t)\pi(t)M=M$, the same reasoning shows that $(\Res_H^G(\pi),\pi(t)M)$ is an irreducible representation of $H$. In particular, since $M^\perp$ is a closed $\pi(H)$-invariant subspace of $\pi(t)M$ and since $(\pi(t)M)^{\perp}$ is a closed $\pi(H)$-invariant subspace of $M$ this proves that $\Hr{\pi}=M\oplus \pi(t)M$. Now let $(\sigma, \Hr{\sigma})=(\Res_H^G(\pi), M)$ and notice that the unitary operator $\pi(t):\pi(t)M\rightarrow M$ intertwines $(\Res_H^G(\pi),\pi(t)M)$ and $(\sigma^t,\Hr{\sigma})$ since for every $\xi \in \pi(t)M$ and for every $h\in H$, we have
		\begin{equation*}
		\begin{split}
		\pi(t)\lb \Res_H^G(\pi)(h)\rb\xi&=\pi(t)\pi(h)\pi(t^{-1})\pi(t)\xi=\pi(tht^{-1})\pi(t)\xi \\
		&=\lb \Res_H^G(\pi)(tht^{-1})\rb \pi(t)\xi =\sigma^t(h)\pi(t)\xi.
		\end{split}
		\end{equation*} 
		This proves as desired that $\Res_H^G(\pi)\cong \sigma \oplus \sigma^t$ for some irreducible representations $\sigma$ of $H$. 
	\end{proof}
		\begin{lemma}\label{Lemma equivalentce criterion between sigma et sigma g et pi et pitau}
		Let $\pi$ be an irreducible representation of $G$ such that $\pi \cong \pi^\tau$. Then, $\Res_H^G(\pi)\cong\sigma \oplus \sigma^t$ for some irreducible representations $\sigma$ of $H$.
	\end{lemma}
	\begin{proof}
		Lemma \ref{Lemma les rep de G se decompose ou pas en irred si H est d'indice 2} ensures that $\Res_H^G(\pi)$ is either irreducible or split as desired. Suppose for a contradiction that $\sigma= \Res_H^G(\pi)$ is irreducible and let $\mathcal{U}: \Hr{\pi}\rightarrow \Hr{\pi}$ be the unitary operator intertwining $\pi$ and $\pi^\tau$. Notice that $\pi$ and $\sigma$ have the same representation space $\Hr{\pi}$. Furthermore, for every $h\in H$, we have that $\pi(h)=\pi^\tau(h)=\sigma(h)$. In particular, $\mathcal{U}$ is a unitary operator that intertwines $\sigma$ with it self. Since $\sigma$ is irreducible this implies that $\mathcal{U}$ is a scalar multiple of the identity. However, this is impossible since for every $h\in H$ and every $\xi \in \Hr{\pi}$ we have 
		\begin{equation*}
		\mathcal{U}\pi(th)\xi=\pi^{\tau}(th)\mathcal{U}\xi=-\pi(th)\mathcal{U}\xi.
		\end{equation*} 
		We obtain as desired that $\Res_H^G(\pi)\cong\sigma \oplus\sigma^t$ for some irreducible representation $\sigma$ of $H$ when $\pi\cong \pi^\tau$.
	\end{proof}
	We now recall the weak version of Frobenius reciprocity that will be used for the rest of the proof of Theorem \ref{les rep dun group loc compact par rapport a celle d'un de ses sous groupes}.
	\begin{theorem}[{\cite[Corollary $1$ of Theorem $3.8$]{Mackey1976theory}} ]\label{criterion of mackey weak frobenius reciprocity}
		Let $G$ be a locally compact group and let $H\leq G$ be a closed subgroup of $G$. Then, for every representation $\pi$ of $G$ and every representation $\sigma$ of $H$ we have
		\begin{equation*}
		{\rm I}(\Ind_H^G(\sigma),\pi)\leq {\rm I} (\sigma, \Res_H^G(\pi)),
		\end{equation*}
		where ${\rm I}(\pi_1,\pi_2)$ is the dimension of the space of intertwining operators between the two representations $\pi_1$ and $\pi_2$. Furthermore, if the index of $H$ in $G$ is finite, this relation becomes an equality.
	\end{theorem}
	\begin{proposition}\label{prop des induite qui se split ou pas}
		Let $\sigma$ be an irreducible representation of $H$. Then, we have $$\Res_H^G(\Ind_H^G(\sigma))\cong\sigma \oplus \sigma^t.$$ Furthermore, the followings hold:
		\begin{itemize}
			\item $\sigma\not\cong \sigma^t$ if and only if $\Ind_H^G(\sigma)$ is an irreducible representation of $G$ and in that case $\Ind_H^G(\sigma)\cong\Ind_H^G(\sigma)^\tau$.
			\item $\sigma \cong \sigma^t$ if and only if $\Ind_H^G(\sigma)\cong \pi \oplus \pi^\tau$ for some irreducible representation $\pi$ of $G$.
		\end{itemize}
	\end{proposition}
	\begin{proof}
		We start by showing that $\Res_H^G(\Ind_H^G(\sigma))\cong\sigma \oplus \sigma^t$. We set $$\mathcal{L}=\{\varphi\in \Ind_H^G(\Hr{\sigma}) \lvert \mbox{supp}(\varphi)\subseteq H\}\mbox{ and }\mathcal{L}^t=\{\varphi\in \Ind_H^G(\Hr{\sigma}) \lvert \mbox{supp}(\varphi)\subseteq Ht\}.$$ By definition of $\Ind_H^G(\Hr{\sigma})$ and since $G= H\sqcup Ht$, it is clear that $$\Ind_H^G(\Hr{\sigma})=\mathcal{L}\oplus \mathcal{L}^t.$$
		Now, notice that $\mathcal{U}: \mathcal{L}\rightarrow \Hr{\sigma}: \varphi \mapsto \varphi(1_G)$ is a unitary operator and that for every $h\in H$ and every $\varphi\in \mathcal{L}$ we have
		\begin{equation*}
		\begin{split}
		\sigma(h)\mathcal{U}\varphi=\sigma(h)\varphi(1_G)=\varphi(h^{-1})=\lb\Ind_H^G(\sigma)(h)\rb\varphi(1_G)=\mathcal{U}\lb \Ind_H^G(\sigma)(h)\rb\varphi. 
		\end{split}
		\end{equation*}
		In particular, this proves that $(\Res_H^G(\Ind_H^G(\sigma)), \mathcal{L})\cong (\sigma,\Hr{\sigma})$.	Similarly, notice that $\mathcal{U}^t: \mathcal{L}^t\rightarrow \Hr{\sigma}: \varphi \mapsto \varphi(t^{-1})$ is a unitary operator and that for every $h\in H$ and every $\varphi\in \mathcal{L}^t$ we have
		\begin{equation*}
		\begin{split}
		\sigma^t(h)\mathcal{U}^t\varphi&=\sigma(tht^{-1})\varphi(t^{-1})=\varphi(t^{-1}th^{-1}t^{-1})\\
		&=\varphi(h^{-1}t^{-1})=\lb \Ind_H^G(\sigma)(h)\rb\varphi(t^{-1})=\mathcal{U}^t\lb \Ind_H^G(\sigma)(h)\rb\varphi.
		\end{split}
		\end{equation*}
		This proves that $(\Res_H^G(\Ind_H^G(\sigma)), \mathcal{L}^t)\cong (\sigma^t,\Hr{\sigma^t})$ and we obtain as desired that $\Res_H^G(\Ind_H^G(\sigma))\cong\sigma \oplus \sigma^t$. 
		
		Since Theorem \ref{criterion of mackey weak frobenius reciprocity} ensures that
		$${\rm I}(\Ind_H^G(\sigma),\Ind_H^G(\sigma))={\rm I}(\sigma,\Res_H^G\big(\Ind_H^G(\sigma)\big)$$
		this implies that $\Ind_H^G(\sigma)$ is irreducible (that is ${\rm I}(\Ind_H^G(\sigma),\Ind_H^G(\sigma))=1$) if and only if $\sigma\not \cong \sigma^t$. Furthermore, in that case, Theorem \ref{criterion of mackey weak frobenius reciprocity} ensures that 
		\begin{equation*}
		\begin{split}
		{\rm I}\big(\Ind_H^G(\sigma),\Ind_H^G(\sigma)^\tau\big)&={\rm I}\big(\sigma,\Res_H^G\big(\Ind_H^G(\sigma)^\tau\big)\big)\\
		&={\rm I}\big(\sigma,\Res_H^G\big(\Ind_H^G(\sigma)\big)\big)\\
		&={\rm I}\big(\Ind_H^G(\sigma),\Ind_H^G(\sigma)\big)=1
		\end{split}
		\end{equation*}
		which proves that $\Ind_H^G(\sigma)\cong\Ind_H^G(\sigma)^\tau$ and settles the first case.
		
		On the other hand, Theorem \ref{criterion of mackey weak frobenius reciprocity} ensures that $\sigma\cong \sigma^t$ if and only if $\Ind_H^G(\sigma)$ is not irreducible.In that case, notice that $\Ind_H^G(\Hr{\sigma})$ must split as a sum of two non-zero closed $G$-invariant subspaces $M$ and $M'$. On the other hand, since $\Res_H^G\big(\Ind_H^G(\sigma)\big)$ splits as a sum of two irreducible representations of $H$, and since every $G$-invariant subspace is $H$-invariant, $M$ and $M'$ do not admit any proper invariant subspaces. This proves that $\Ind_H^G(\sigma)\cong \pi\oplus\pi'$ for some irreducible representations $\pi$ and $\pi'$ of $G$. On the other hand, since $\Res_H^G(\pi)=\Res_H^G(\pi^\tau)$, Theorem \ref{criterion of mackey weak frobenius reciprocity} ensures that
		$${\rm I}(\Ind_H^G(\sigma),\pi)={\rm I}(\sigma,\Res_H^G(\pi))={\rm I}(\sigma,\Res_H^G(\pi^\tau))={\rm I}(\Ind_H^G(\sigma),\pi^\tau)$$
		for every irreducible representation $\pi$ of $G$. In particular, if $\pi\not \cong \pi^\tau$, we obtain that $\Ind_H^G(\sigma)\cong \pi \oplus \pi^\tau$. On the other hand, if $\pi\cong \pi^\tau$, notice from Lemma \ref{Lemma equivalentce criterion between sigma et sigma g et pi et pitau} that $\Res_H^G(\pi)\cong \sigma\oplus \sigma^t$. Hence, since $\sigma\cong \sigma^t$, Theorem \ref{criterion of mackey weak frobenius reciprocity} implies that
		\begin{equation*}
		\begin{split}
		{\rm I}(\Ind_H^G(\sigma),\pi)={\rm I}(\sigma,\Res_H^G(\pi))={\rm I}(\sigma,\sigma \oplus \sigma^t)>1
		\end{split}
		\end{equation*}
		which proves that $\Ind_H^G(\sigma)\cong \pi \oplus \pi \cong \pi \oplus \pi^\tau$.
	\end{proof}
	The first part of Theorem \ref{les rep dun group loc compact par rapport a celle d'un de ses sous groupes} follows from Lemma \ref{Lemma les rep de G se decompose ou pas en irred si H est d'indice 2}, Lemma \ref{Lemma equivalentce criterion between sigma et sigma g et pi et pitau}, Proposition \ref{prop des induite qui se split ou pas} and from the impossibility to have simultaneously that $\pi\cong \pi^\tau$ and that $\sigma\cong \sigma^t$. Indeed, if $\pi\cong \pi^\tau$, Lemma \ref{Lemma equivalentce criterion between sigma et sigma g et pi et pitau} ensures that $\Res_H^G(\pi)\cong \sigma\oplus \sigma^t$. However, if $\sigma \cong \sigma^t$, Proposition \ref{prop des induite qui se split ou pas} ensures that $\Ind_H^G(\sigma)\cong \pi \oplus \pi^\tau$. In particular, if those conditions were satisfied simultaneously one would obtain that $$\Res_H^G\big(\Ind_H^G(\sigma)\big)\cong \sigma \oplus \sigma^t \oplus \sigma \oplus \sigma^t\cong 4 \sigma$$
	which impossible since Proposition \ref{prop des induite qui se split ou pas} ensures that $$\Res_H^G(\Ind_H^G(\sigma))\cong\sigma \oplus \sigma^t\cong 2\sigma.$$
	
	\noindent The following result completes the proof of Theorem \ref{les rep dun group loc compact par rapport a celle d'un de ses sous groupes}.
	\begin{lemma}
		Every irreducible representation $\pi$ of $G$ satisfies $\pi\leq \Ind_H^G(\sigma)$ for some irreducible representation $\sigma$ of $H$ and every irreducible representation $\sigma$ of $H$ satisfies $\sigma\leq \Res_H^G(\pi)$ for some irreducible representation $\pi$ of $G$. 
	\end{lemma}
	\begin{proof}
		Let $\pi$ be an irreducible representation of $G$ and let us show that $\pi\leq \Ind_H^G(\sigma)$ for some irreducible representation $\sigma$ of $H$. Notice from Theorem \ref{criterion of mackey weak frobenius reciprocity} that
		$${\rm I}(\Ind_H^G\big(\Res_H^G(\pi)\big), \pi)={\rm I}(\Res_H^G(\pi),\Res_H^G(\pi))\geq1$$
		which proves that $\pi\leq \Ind_H^G\big(\Res_H^G(\pi)\big)$. If $\Res_H^G(\pi)$ is irreducible, the result follows trivially. On the other hand, if $\Res_H^G(\pi)$ is not irreducible, Lemma \ref{Lemma les rep de G se decompose ou pas en irred si H est d'indice 2} ensures that $\Res_H^G(\pi)\cong \sigma\oplus \sigma^t$ for some irreducible representation $\sigma$ of $H$. In particular, since $\pi \leq \Ind_H^G\big(\Res_H^G(\pi)\big) \cong \Ind_H^G(\sigma)\oplus\Ind_H^G(\sigma^t)$ we obtain either that $\pi \leq \Ind_H^G(\sigma)$ or that $\pi \leq \Ind_H^G(\sigma^t)$. 
		
		Now, let $\sigma$ be an irreducible representation of $H$ and let us show that $\sigma\leq \Res_H^G(\pi)$ for some irreducible representation $\pi$ of $G$. Notice from Theorem \ref{criterion of mackey weak frobenius reciprocity} that
		$${\rm I}(\sigma,\Res_H^G\big(\Ind_H^G(\sigma)\big))={\rm I}(\Ind_H^G(\sigma), \Ind_H^G(\sigma))\geq1$$
		which proves that $\sigma\leq \Res_H^G\big(\Ind_H^G(\sigma)\big)$. If $\Ind_H^G(\sigma)$ is irreducible, the result follows trivially. On the other hand, if $\Ind_H^G(\pi)$ is not irreducible, Proposition \ref{prop des induite qui se split ou pas} ensures that $\Ind_H^G(\sigma)\cong \pi\oplus \pi^\tau$ for some irreducible representation $\pi$ of $G$. In particular, since $\pi \leq \Ind_H^G\big(\Res_H^G(\pi)\big) \cong \Ind_H^G(\sigma)\oplus\Ind_H^G(\sigma^t)$ we obtain either that $\pi \leq \Ind_H^G(\sigma)$ or that $\pi \leq \Ind_H^G(\sigma^t)$.
	\end{proof}
	Finally, as a consequence of Theorem \ref{theorem simple radu groups are unifmroly admissible} and of the correspondence provided by Theorem \ref{les rep dun group loc compact par rapport a celle d'un de ses sous groupes}, the following lemma ensures that every Radu group on a $(d_0,d_1)$-semi-regular tree with $d_0,d_1\geq 6$ is uniformly admissible.
	\begin{lemma}\label{lemma G admissible iff  H is admissible}
		Let $G$ be a totally disconnected locally compact group and let $H$ be a closed subgroup of index $2$. Then, $G$ is uniformly admissible if and only if $H$ is uniformly admissible.
	\end{lemma}
	\begin{proof}
		Suppose that $G$ is uniformly admissible and let $K$ be a compact open subgroup of $H$. Since $H$ has index $2$ in $G$, it is a clopen subgroup of $G$ whcih implies that $K$ is a compact open subgroup of $G$. Since $G$ is uniformly admissible, there exists a constant $k_K\in \N$ such that $\dim(\Hr{\pi}^K)\leq k_K$ for every irreducible representation $\pi$ of $G$. Let $\sigma$ be an irreducible representation of $H$. Theorem \ref{les rep dun group loc compact par rapport a celle d'un de ses sous groupes} ensures that $\Ind_H^G(\sigma)$ is either irreducible or splits as a sum of two irreducible representations of $G$. On the other hand, notice from Theorem \ref{criterion of mackey weak frobenius reciprocity} that $${\rm I}(\sigma,\Res_H^G(\Ind_H^G(\sigma)))={\rm I}(\Ind_H^G(\sigma),\Ind_H^G(\sigma))\geq 1$$ which implies that $\sigma\leq \Res_H^G(\Ind_H^G(\sigma))$. All together, this proves that $$\dim(\Hr{\sigma}^K)\leq \dim\big(\Hr{\Ind_H^G(\sigma)}^K\big)\leq 2k_K$$ and $H$ is uniformly admissible. 
		
		Suppose now that $H$ is uniformly admissible and let $K$ be a compact open subgroup of $G$. Since $H$ has index $2$ in $G$, it is a clopen subgroup of $G$ and $K\cap H$ is a compact open subgroup of $H$. Since $H$ is uniformly admissible, this implies the existence of a constant $k_{K\cap H}$ such that $\dim(\Hr{\sigma})\leq k_{k_{K\cap H}}$ for every irreducible representation $\sigma$ of $H$. Furthermore, Theorem \ref{les rep dun group loc compact par rapport a celle d'un de ses sous groupes} ensures that $\Res_H^G(\pi)$ is either an irreducible representation of $G$ or splits as a direct sum of $2$ irreducible representations of $H$. This implies that $$\dim(\Hr{\pi}^K)\leq \dim(\Hr{\pi}^{K\cap H})= \dim\big(\Hr{\Res_H^G(\pi)}^{K\cap H}\big)\leq 2k_{K\cap H}.$$
		Hence, $G$ is uniformly admissible. 
	\end{proof}
	\end{appendices}
\clearpage
\newpage
\addcontentsline{toc}{section}{Bibliography}
\bibliographystyle{alpha}
\bibliography{bibliography}	
\end{document}